\documentclass[a4paper,10pt]{amsart}
\usepackage{pb-diagram}
\usepackage{xypic}
\usepackage{amscd}
\usepackage{bm}
\usepackage{amsmath,amssymb,amscd,bbm,amsthm,mathrsfs,hyperref}
\usepackage{latexsym}
\usepackage{amsfonts}
\usepackage{amsthm}
\usepackage{indentfirst}
 \newtheorem{thm}{Theorem}[section]
 \newtheorem{cor}[thm]{Corollary}
 \newtheorem{lem}[thm]{Lemma}
 \newtheorem{prop}[thm]{Proposition}
 \newtheorem{defn}[thm]{Definition}
 
 \newtheorem{rem}[thm]{Remark}

 \def\k{\mathbbm{k}}

 \newcommand{\Hom}{\mathrm{Hom}}

\title{Isomorphism problem and homological properties of DG free algebras}
\author{X.-F. Mao}
\address{Department of Mathematics, Shanghai University, Shanghai 200444, China}
\email{xuefengmao@shu.edu.cn}

\author{J.-F. Xie}
\address{Department of Mathematics, Shanghai University, Shanghai 200444, China}
\email{jianfengxie@yahoo.com}

\author{Y.-N. Yang}
\address{Department of Mathematics, Shanghai University, Shanghai 200444, China}
\email{mooly@shu.edu.cn}

\author{Almire. Abla}
\address{Department of Mathematics, College of Mathematics and Statistics, Kashgar University, Kashgar, Xinjiang, 844006, China}
\email{657731648@qq.com}

\date{}

\begin{document}

\maketitle \def\abstactname{abstact}
\begin{abstract}
A differential graded (DG for short) free algebra $\mathcal{A}$ is a connected cochain DG algebra
such that its underlying graded algebra is $$\mathcal{A}^{\#}=\k\langle x_1,x_2,\cdots, x_n\rangle,\,\, \text{with}\,\, |x_i|=1,\,\, \forall i\in \{1,2,\cdots, n\}.$$ We prove that the differential structures on DG free algebras are in one to one correspondence with  the set of crisscross ordered $n$-tuples of $n\times n$ matrixes. We also give a criterion to judge whether two DG free algebras are isomorphic.

As an application, we consider the case of $n=2$. Based on the isomorphism classification, we compute the cohomology graded algebras of non-trivial DG free algebras with $2$ generators, and show that all those non-trivial DG free algebras are Koszul and Calabi-Yau.
  \end{abstract}

\maketitle

\section{introduction}

 Throughout this paper, $\k$ is an algebraically closed field of characteristic $0$. The construction of some interesting  DG algebras is important in DG homological algebra. One always depends on the computations of some specific examples to deduce general rules on DG algebras.  Recall that a cochain DG algebra is
a $\Bbb{Z}$-graded
$\k$-algebra $\mathcal{A}$ together with a degree one $\k$-linear map $\partial_{\mathcal{A}}$ from $\mathcal{A}$  to itself
such that $\partial_{\mathcal{A}}\circ \partial_{\mathcal{A}}  = 0$ and
$$\partial_{\mathcal{A}}(ab) = \partial_{\mathcal{A}} (a)b + (-1)^{|a|}a\partial_{\mathcal{A}} (b),$$
for all graded elements $a, b\in \mathcal{A}$.  By definition, the graded algebra structure and the differential structure are two  essential factors of  a DG algebra.
If one regard a DG algebra $\mathcal{A}$ as a living thing, then the underlying graded algebra $\mathcal{A}^{\#}$  and the differential $\partial_{\mathcal{A}}$ are its body and soul, respectively. An efficient way to create meaningful DG algebras is to select some well known regular graded algebras as bodies, and then inject reasonable differential structures into it.
 In the literature, there has been some attempt for this. Especially, we have the following list:

 \begin{center}
\begin{tabular}{|l|l|}
  \hline
  Reference & the chosen underlying graded algebra \\
  \hline
  \cite{Mao} & Artin-Schelter regular algebra of global dimension $2$ \\
  \cite{MHLX} & graded down-up algebra\\
  \cite{MGYC} & polynomial algebra \\
  \hline
\end{tabular}
\end{center}
In this paper, we choose free graded associative algebras as the underlying graded algebras. We say that a cochain DG algebra $\mathcal{A}$ is a DG free algebra if $\mathcal{A}^{\#}$ is the free algebra $\k\langle x_1,x_2,\cdots,x_n\rangle$ with each $|x_i|=1$.  A cochain DG algebra $\mathcal{A}$ is called
non-trivial if $\partial_{\mathcal{A}}\neq 0$, and $\mathcal{A}$ is said to be connected if its underlying graded algebra $\mathcal{A}^{\#}$ is a connected  graded algebra. Obviously, any DG free algebra is a connected cochain DG algebra. In order to study DG free algebras systematically,
we describe all possible differential structures on DG free algebras by the following theorem (see Theorem \ref{diffstr}). \\
\begin{bfseries}
Theorem \ A.
\end{bfseries}
Let $(\mathcal{A},\partial_{\mathcal{A}})$ be a connected cochain DG algebra such that $\mathcal{A}^{\#}$ is a free graded algebra $\k\langle x_1,x_2,\cdots,x_n\rangle $ with $|x_i|=1$, for any $i\in \{1,2,\cdots,n\}$.
Then there exist a crisscross ordered $n$-tuple $(M^1,M^2,\cdots, M^n)$ of $n\times n$ matrixes such that
 $\partial_{\mathcal{A}}$ is defined  by
$$\partial_{\mathcal{A}}(x_i)=(x_1,x_2,\cdots, x_n)M^i\left(
                         \begin{array}{c}
                           x_1 \\
                           x_2  \\
                           \vdots  \\
                           x_n
                         \end{array}
                       \right). $$  Conversely, given a crisscross ordered $n$-tuple $(M^1,M^2,\cdots, M^n)$ of $n\times n$ matrixes, we can define a differential $\partial$ on $\k\langle x_1,x_2,\cdots,x_n\rangle$ by
\begin{align*}\partial(x_i)=(x_1,x_2,\cdots, x_n)M^i\left(
                         \begin{array}{c}
                           x_1 \\
                           x_2  \\
                           \vdots  \\
                           x_n
                         \end{array}
                       \right), \forall i\in\{1,2,\cdots, n\}
                       \end{align*}
                        such that $(\k\langle x_1,x_2,\cdots, x_n\rangle, \partial)$ is a cochain DG algebra.

                        One sees the definition of crisscross ordered $n$-tuple $(M^1,M^2,\cdots, M^n)$ of $n\times n$ matrixes in Definition \ref{crisscrossed}.
Theorem A indicates that DG free algebras have plenty of differential structures.
To study DG free algebras systematically, we should consider the isomorphism problem first. This paper gives a criterion by the following theorem (see Theorem \ref{isom}). \\
\begin{bfseries}
Theorem \ B.
\end{bfseries}
Let $\mathcal{A}$ and $\mathcal{B}$ be two DG free algebras such that $$ \mathcal{A}^{\#}=\k\langle x_1,x_2,\cdots, x_n\rangle, \quad \mathcal{B}^{\#}=\k\langle y_1,y_2,\cdots, y_n\rangle, $$
with each $|x_i|=|y_i|=1$. Assume that $\partial_{\mathcal{A}}$ and $\partial_{\mathcal{B}}$ are defined by crisscross ordered $n$-tuples $(M^1,\cdots, M^n)$ and $(N^1,\cdots, N^n)$,  respectively. Then $\mathcal{A}\cong \mathcal{B}$ if and only if there exists $A=(a_{ij})_{n\times n}\in \mathrm{GL}_n(\k)$ such that
$$(a_{ij}E_n)_{n^2\times n^2}\left(
                         \begin{array}{c}
                           N_1 \\
                           N_2  \\
                           \vdots  \\
                           N_n
                         \end{array}
                       \right)=\left(
                         \begin{array}{c}
                           A^TM^1A \\
                           A^TM^2A  \\
                           \vdots  \\
                           A^TM^nA
                         \end{array}
                       \right). $$

In general, the properties of a DG algebra are determined by the joint effects of its underlying graded algebra structure and differential structure. However, it is feasible, at least in some special
cases, to judge some properties of a DG algebra $\mathcal{A}$ from $\mathcal{A}^{\#}$.  For example, it is shown in \cite{Mao} that
 a connected cochain DG algebra $\mathcal{B}$ is Gorenstein if its underlying graded algebra $\mathcal{B}^{\#}$ is an Artin-Schelter regular algebra of global dimension $2$. In \cite{MHLX},  all non-trivial Noetherian DG down-up algebras are proved Calabi-Yau. Recently,  DG polynomial algebras with degree one generators are systematically studied in \cite{MGYC}. It is proved that any non-trivial DG polynomial algebra is Calabi-Yau and a trivial DG polynomial algebra is Calabi-Yau if and only if it is generated by odd number of generators. In this paper, we attempt to
 figure out homological properties of DG free algebras.
  We show that any trivial DG free algebra is not Gorenstein but homologically smooth (see Proposition \ref{zerodiff}).  When it comes to non-trivial cases, things become more complicated since
 it seems not feasible to classify all the isomorphism classes of DG free algebras when $n\ge 3$.
 As a consolation, we completely solve the case of $n=2$. We classify the isomorphism classes of DG free algebras with $2$ degree one generators (see Proposition \ref{nonsym}, Proposition \ref{symone} and Proposition \ref{symtwo}). And we reach the following interesting conclusion (see Theorem \ref{finresult}). \\
 \begin{bfseries}
Theorem \ C.
\end{bfseries}
 Let $\mathcal{A}$ be a DG free algebra with $2$ degree one generators. Then $\mathcal{A}$ is a Koszul Calabi-Yau DG algebra if and only if $\partial_{\mathcal{A}}\neq 0$.

\section{differential structures on dg free algebras}
In this section, we will study the differential structures on DG free algebras. For this, we introduce the definition of crisscross ordered $n$-tuple of $n\times n$ matrixes first.
\begin{defn}\label{crisscrossed}
{\rm Let $(M^1,M^2,\cdots, M^n)$ be an ordered $n$-tuple of $n\times n$ matrixes with each
\begin{align*}
M^i=(c^i_1,c^i_2,\cdots,c^i_n)=\left(
                         \begin{array}{c}
                           r^i_1 \\
                           r^i_2  \\
                           \vdots  \\
                           r^i_n
                         \end{array}
                       \right), i=1,2,\cdots,n.
\end{align*}
We say that $(M^1,M^2,\cdots, M^n)$ is crisscross if
$$
\sum\limits_{k=1}^n[c_j^kr_k^i-c_k^ir_j^k]=(0)_{n\times n}, \forall i,j\in \{1,2,\cdots,n\}.
$$

}
\end{defn}
In the rest of this section, we will reveal the close relations between crisscross ordered $n$-tuples of $n\times n$ matrixes and the differential structure of DG free algebras. The following lemma will be used in subsequent computations.

\begin{lem}\label{basiclem}
Let $N^1,N^2,\cdots, N^n$ be $n\times n$ matrixes such that
\begin{align*}
(x_1,x_2,\cdots,x_n)[N^1x_1+N^2x_2+\cdots + N^nx_n]\left(
                         \begin{array}{c}
                           x_1 \\
                           x_2  \\
                           \vdots  \\
                           x_n
                         \end{array}
                       \right)=0
\end{align*}
in $\k\langle x_1,x_2,\cdots,x_n\rangle$. Then each $N^i=(0)_{n\times n},  \forall i\in \{1,2,\cdots,n\}$.
\end{lem}
\begin{proof}
Let $N^i=(a^i_{jk})_{n\times n}$, $i=1,2,\cdots, n$.  By the assumption, we have
\begin{align*}
0=&(x_1,x_2,\cdots,x_n)[N^1x_1+N^2x_2+\cdots + N^nx_n]\left(
                         \begin{array}{c}
                           x_1 \\
                           x_2  \\
                           \vdots  \\
                           x_n
                         \end{array}
                       \right) \\
=&  (x_1,x_2,\cdots,x_n)\left(
                          \begin{array}{ccccc}
                            \sum\limits_{i=1}^na^i_{11}x_i & \cdots & \sum\limits_{i=1}^na_{1k}^ix_i & \cdots & \sum\limits_{i=1}^na_{1n}^ix_i \\
                            \vdots & \vdots & \vdots & \vdots & \vdots \\
                            \sum\limits_{i=1}^n a_{j1}^ix_i & \cdots & \sum\limits_{i=1}^na_{jk}^ix_i & \cdots & \sum\limits_{i=1}^na_{jn}^ix_i \\
                            \vdots & \vdots & \vdots & \vdots & \vdots \\
                            \sum\limits_{i=1}^na_{n1}^ix_i & \cdots & \sum\limits_{i=1}^na_{nk}^ix_i & \cdots & \sum\limits_{i=1}^na_{nn}^ix_i \\
                          \end{array}
                        \right)\left(
                         \begin{array}{c}
                           x_1 \\
                           x_2  \\
                           \vdots  \\
                           x_n
                         \end{array}
                       \right) \\
=&(\sum\limits_{j=1}^n\sum\limits_{i=1}^na_{j1}^ix_jx_i,\sum\limits_{j=1}^n\sum\limits_{i=1}^na_{j2}^ix_jx_i,\cdots,\sum\limits_{j=1}^n\sum\limits_{i=1}^na_{jn}^ix_jx_i)\left(
                         \begin{array}{c}
                           x_1 \\
                           x_2  \\
                           \vdots  \\
                           x_n
                         \end{array}
                       \right) \\
=&\sum\limits_{k=1}^n\sum\limits_{j=1}^n\sum\limits_{i=1}^na_{jk}^ix_jx_ix_k \quad \text{in}\quad \k\langle x_1,x_2,\cdots, x_n\rangle.
\end{align*}
This implies that $a_{jk}^i=0,\forall i,j,k\in \{1,2,\cdots, n\}$.  So $N^i=(0)_{n\times n}, i=1,2,\cdots,n$.
\end{proof}

\begin{thm}\label{diffstr}
Let $(\mathcal{A},\partial_{\mathcal{A}})$ be a connected cochain DG algebra such that $\mathcal{A}^{\#}$ is a free graded algebra $\k\langle x_1,x_2,\cdots,x_n\rangle $ with $|x_i|=1$, for any $i\in \{1,2,\cdots,n\}$.
Then there exist a crisscross ordered $n$-tuple $(M^1,M^2,\cdots, M^n)$ of
$n\times n$ matrixes such that
 $\partial_{\mathcal{A}}$ is defined  by
$$\partial_{\mathcal{A}}(x_i)=(x_1,x_2,\cdots, x_n)M^i\left(
                         \begin{array}{c}
                           x_1 \\
                           x_2  \\
                           \vdots  \\
                           x_n
                         \end{array}
                       \right). $$  Conversely, given a crisscross ordered $n$-tuple $(M^1,M^2,\cdots, M^n)$ of
$n\times n$ matrixes,  we can define a differential $\partial$ on $\k\langle x_1,x_2,\cdots,x_n\rangle$ by
\begin{align*}\partial(x_i)=(x_1,x_2,\cdots, x_n)M^i\left(
                         \begin{array}{c}
                           x_1 \\
                           x_2  \\
                           \vdots  \\
                           x_n
                         \end{array}
                       \right), \forall i\in\{1,2,\cdots, n\}
                       \end{align*}
                        such that $(\k\langle x_1,x_2,\cdots, x_n\rangle, \partial)$ is a cochain DG algebra.

\end{thm}
\begin{proof}
Since the differential $\partial_{\mathcal{A}}$ of $\mathcal{A}$ is a $\k$-linear map of degree $1$, we may let
$$
\partial_{\mathcal{A}}(x_i) = (x_1,x_2,\cdots, x_n)M^i\left(
                         \begin{array}{c}
                           x_1 \\
                           x_2  \\
                           \vdots  \\
                           x_n
                         \end{array}
                       \right), $$
where \begin{align*}
M^i=(c^i_1,c^i_2,\cdots,c^i_n)=\left(
                         \begin{array}{c}
                           r^i_1 \\
                           r^i_2  \\
                           \vdots  \\
                           r^i_n
                         \end{array}
                       \right), \forall i\in\{1,2,\cdots, n\}.
\end{align*}
Since  $(\mathcal{A},\partial_{\mathcal{A}})$ is a cochain DG
algebra, $\partial_{\mathcal{A}}$ satisfies the
Leibniz rule and
\begin{align}\label{eqs}
\partial_{\mathcal{A}}\circ \partial_{\mathcal{A}}(x_i) = 0, \forall i\in\{1,2,\cdots,n\}.
\end{align}
Hence
\begin{align*}
&\quad\quad 0=\partial_A\circ\partial_A(x_i)=\partial_A[(x_1,x_2,\cdots, x_n)M^i\left(
                         \begin{array}{c}
                           x_1 \\
                           x_2  \\
                           \vdots  \\
                           x_n
                         \end{array}
                       \right)]\\
                    & =(\partial_A(x_1),\partial_A(x_2),\cdots,\partial_A(x_n)) M^i\left(
                         \begin{array}{c}
                           x_1 \\
                           x_2  \\
                           \vdots  \\
                           x_n
                         \end{array}
                       \right)  - (x_1,x_2,\cdots, x_n)M^i \left(
                         \begin{array}{c}
                           \partial_A(x_1) \\
                           \partial_A(x_2)  \\
                           \vdots  \\
                           \partial_A(x_n)
                         \end{array}
                       \right)\\
                   &   =(x_1,x_2,\cdots, x_n) [M^1\left(
                         \begin{array}{c}
                           x_1 \\
                           x_2  \\
                           \vdots  \\
                           x_n
                         \end{array}
                       \right),M^2\left(
                         \begin{array}{c}
                           x_1 \\
                           x_2  \\
                           \vdots  \\
                           x_n
                         \end{array}
                       \right),\cdots,M^n\left(
                         \begin{array}{c}
                           x_1 \\
                           x_2  \\
                           \vdots  \\
                           x_n
                         \end{array}
                       \right)]M^i\left(
                         \begin{array}{c}
                           x_1 \\
                           x_2  \\
                           \vdots  \\
                           x_n
                         \end{array}
                       \right)\\
& \quad - (x_1,x_2,\cdots, x_n)M^i\left(
                         \begin{array}{c}
                           (x_1,x_2,\cdots,x_n)M^1 \\
                           (x_1,x_2,\cdots,x_n)M^2  \\
                           \vdots  \\
                           (x_1,x_2,\cdots,x_n)M^n
                         \end{array}
                       \right)\left(
                         \begin{array}{c}
                           x_1 \\
                           x_2  \\
                           \vdots  \\
                           x_n
                         \end{array}
                       \right)\\
&=(x_1,x_2,\cdots,x_n)(\sum\limits_{j=1}^nc_j^1x_j,\sum\limits_{j=1}^nc_j^2x_j,\cdots,\sum\limits_{j=1}^nc_j^nx_j)\left(
                         \begin{array}{c}
                           r^i_1 \\
                           r^i_2  \\
                           \vdots  \\
                           r^i_n
                         \end{array}
                       \right)\left(
                         \begin{array}{c}
                           x_1 \\
                           x_2  \\
                           \vdots  \\
                           x_n
                         \end{array}
                       \right)\\
&\quad -(x_1,x_2,\cdots,x_n)(c^i_1,c^i_2,\cdots,c_n^i)\left(
                         \begin{array}{c}
                           \sum\limits_{j=1}^nx_jr_j^1 \\
                           \sum\limits_{j=1}^nx_jr_j^2  \\
                           \vdots  \\
                           \sum\limits_{j=1}^nx_jr_j^n
                         \end{array}
                       \right)\left(
                         \begin{array}{c}
                           x_1 \\
                           x_2  \\
                           \vdots  \\
                           x_n
                         \end{array}
                       \right)\\
&=(x_1,x_2,\cdots,x_n)[\sum\limits_{k=1}^n\sum\limits_{j=1}^nc_j^kx_jr_k^i-\sum\limits_{k=1}^n\sum\limits_{j=1}^nc_k^ix_jr_j^k]\left(
                         \begin{array}{c}
                           x_1 \\
                           x_2  \\
                           \vdots  \\
                           x_n
                         \end{array}
                       \right)\\
&=(x_1,x_2,\cdots,x_n)[\sum\limits_{j=1}^n\sum\limits_{k=1}^n(c_j^kr_k^i-c_k^ir_j^k)x_j]\left(
                         \begin{array}{c}
                           x_1 \\
                           x_2  \\
                           \vdots  \\
                           x_n
                         \end{array}
                       \right), \forall i\in\{1,2,\cdots, n\}.
\end{align*}
By Lemma \ref{basiclem}, we have $$\sum\limits_{k=1}^n(c_j^kr_k^i-c_k^ir_j^k)=(0)_{n\times n}, \forall i, j\in \{1,2,\cdots,n\}.$$ So $(M^1,M^2,\cdots, M^n)$ is a crisscross ordered $n$-tuple of $n\times n$ matrixes.

Conversely, if $(M^1,M^2,\cdots, M^n)$ is a crisscross ordered $n$-tuple of $n\times n$ matrixes with each
\begin{align*}
M^i=(c^i_1,c^i_2,\cdots,c^i_n)=\left(
                         \begin{array}{c}
                           r^i_1 \\
                           r^i_2  \\
                           \vdots  \\
                           r^i_n
                         \end{array}
                       \right),
\end{align*} then we have $$\sum\limits_{k=1}^n(c_j^kr_k^i-c_k^ir_j^k)=(0)_{n\times n}, \forall i, j\in \{1,2,\cdots,n\}.$$
We can define a differential $\partial$ on $\k\langle x_1,x_2,\cdots,x_n\rangle$ by
$$\partial(x_i)= (x_1,x_2,\cdots, x_n)M^i\left(
                         \begin{array}{c}
                           x_1 \\
                           x_2  \\
                           \vdots  \\
                           x_n
                         \end{array}
                       \right), i=1,2,\cdots, n,$$
such that $(\k\langle x_1,x_2,\cdots,x_n\rangle,\partial)$ is a cochain DG algebra, since one can check as above that $$
\partial \circ \partial (x_i) = 0, \forall i\in\{1,2,\cdots,n\},
$$
if $\partial$ satisfies the
Leibniz rule.

\end{proof}
\begin{rem}\label{onetoone}
 Theorem \ref{diffstr} indicates that there is a one to one in correspondence between
                        $$\{\mathcal{A}|\mathcal{A} \,\,\text{is a DG free algebra with} \,\,\mathcal{A}^{\#}=\k\langle x_1,x_2,\cdots, x_n\rangle \}$$
and the set
$$\{(M^1,\cdots, M^n)|(M^1,\cdots, M^n) \,\, \text{is a crisscross ordered}\,\, n\text{-tuple},\,M^i\in M_n(\k)\}.$$
\end{rem}

\section{Isomorphism problems for DG free algebras}
Remark \ref{onetoone} implies that DG free algebras have abundant differential structures. For future systematically studies, one should consider the isomorphism problem of DG free algebras first, since two isomorphic DG free algebras have same homological properties.
 A successful classification work on the isomorphism classes of DG free algebras will efficiently simplify our research.
 We have the following theorem.
\begin{thm}\label{isom}
Let $\mathcal{A}$ and $\mathcal{B}$ be two DG free algebras such that $$ \mathcal{A}^{\#}=\k\langle x_1,x_2,\cdots, x_n\rangle, \quad \mathcal{B}^{\#}=\k\langle y_1,y_2,\cdots, y_n\rangle, $$
with each $|x_i|=|y_i|=1$. By Theorem \ref{diffstr}, $\partial_{\mathcal{A}}$ and $\partial_{\mathcal{B}}$ are defined by crisscrossed $n\times n$ matrixes $M^1, M^2,\cdots, M^n$ and $N^1,N^2,\cdots, N^n$  respectively. Then $\mathcal{A}\cong \mathcal{B}$ if and only if there exists $A=(a_{ij})_{n\times n}\in \mathrm{GL}_n(\k)$ such that
$$(a_{ij}E_n)_{n^2\times n^2}\left(
                         \begin{array}{c}
                           N_1 \\
                           N_2  \\
                           \vdots  \\
                           N_n
                         \end{array}
                       \right)=\left(
                         \begin{array}{c}
                           A^TM^1A \\
                           A^TM^2A  \\
                           \vdots  \\
                           A^TM^nA
                         \end{array}
                       \right). $$
\end{thm}
\begin{proof}
If the DG
algebras $\mathcal{A}\cong \mathcal{B}$, then there exists an
isomorphism $f: \mathcal{A}\to \mathcal{B}$ of DG algebras. Since
$f^1: \mathcal{A}^1 \to \mathcal{B}^1$ is a $\k$-linear
isomorphism, we may let $$\left(\begin{array}{c}
                           f(x_1) \\
                           f(x_2) \\
                           \vdots \\
                           f(x_n)
                         \end{array}
                       \right)= A\left(
                         \begin{array}{c}
                           y_1\\
                           y_2 \\
                           \vdots \\
                           y_n
                         \end{array}
                       \right)$$
for some  $A=(a_{ij})_{n\times n} \in \mathrm{GL}_n(\k)$.
Since $f$ is a chain map, we have $f\circ \partial_{\mathcal{A}}=\partial_{\mathcal{B}}\circ f$.
For any $i\in \{1,2,\cdots,n\}$, we have
\begin{align*}\label{Eq1}\tag{Eq1}
&\quad\quad\quad \partial_{\mathcal{B}}\circ f(x_i)=\partial_{\mathcal{B}}(\sum\limits_{j=1}^na_{ij}y_j)\\
&=\sum\limits_{j=1}^na_{ij}[ (y_1,y_2,\cdots, y_n)N^j\left(
                         \begin{array}{c}
                           y_1 \\
                           y_2  \\
                           \vdots  \\
                           y_n
                         \end{array}
                       \right)]\\
&=\sum\limits_{j=1}^na_{ij}[\sum\limits_{s=1}^n\sum\limits_{t=1}^n n_{st}^jy_sy_t]\\
&=\sum\limits_{j=1}^n\sum\limits_{s=1}^n\sum\limits_{t=1}^na_{ij}n_{st}^jy_sy_t \\
&=\sum\limits_{s=1}^n\sum\limits_{t=1}^n\sum\limits_{j=1}^na_{ij}n_{st}^jy_sy_t
\end{align*}
and
\begin{align*}\label{Eq2}\tag{Eq2}
&\quad \quad\quad\quad f\circ \partial_{\mathcal{A}}(x_i)=f[(x_1,x_2,\cdots, x_n)M^i\left(
                         \begin{array}{c}
                           x_1 \\
                           x_2  \\
                           \vdots  \\
                           x_n
                         \end{array}
                       \right)]\\
             & \quad \quad\quad\quad\quad \quad\quad\quad\quad\quad = f[\sum\limits_{k=1}^n\sum\limits_{l=1}^nm^i_{kl}x_kx_l]\\
             & \quad \quad\quad\quad\quad \quad\quad\quad\quad\quad = \sum\limits_{k=1}^n\sum\limits_{l=1}^nm^i_{kl}f(x_k)f(x_l)\\
             &\quad\quad\quad\quad\quad\quad\quad\quad=\sum\limits_{k=1}^n\sum\limits_{l=1}^nm^i_{kl}(\sum\limits_{s=1}^na_{ks}y_s)(\sum\limits_{t=1}^na_{lt}y_t) \\
             &\quad\quad\quad\quad\quad\quad\quad\quad=\sum\limits_{s=1}^n\sum\limits_{t=1}^n\sum\limits_{k=1}^n\sum\limits_{l=1}^nm^i_{kl}a_{ks}a_{lt}y_sy_t \\
             &\quad\quad\quad\quad\quad\quad\quad\quad=\sum\limits_{s=1}^n\sum\limits_{t=1}^n\sum\limits_{k=1}^n\sum\limits_{l=1}^na_{ks}m^i_{kl}a_{lt}y_sy_t.
\end{align*}
So $ \partial_{\mathcal{B}}\circ f(x_i)=f\circ \partial_{\mathcal{A}}(x_i)$ implies that $$\sum\limits_{j=1}^na_{ij}n_{st}^j=\sum\limits_{k=1}^n\sum\limits_{l=1}^na_{ks}m^i_{kl}a_{lt}, \forall i,s,t\in \{1,2,\cdots, n\}.$$ Then we have
$\sum\limits_{j=1}^na_{ij}N^j=A^TM^iA$, for any $i=1,2,\cdots, n$. Therefore, $$(a_{ij}E_n)_{n^2\times n^2}\left(
                         \begin{array}{c}
                           N_1 \\
                           N_2  \\
                           \vdots  \\
                           N_n
                         \end{array}
                       \right)=\left(
                         \begin{array}{c}
                           A^TM^1A \\
                           A^TM^2A  \\
                           \vdots  \\
                           A^TM^nA
                         \end{array}
                       \right). $$

Conversely, if there exists $A=(a_{ij})_{n\times n}\in \mathrm{GL}_n(\k)$ such that
$$(a_{ij}E_n)_{n^2\times n^2}\left(
                         \begin{array}{c}
                           N_1 \\
                           N_2  \\
                           \vdots  \\
                           N_n
                         \end{array}
                       \right)=\left(
                         \begin{array}{c}
                           A^TM^1A \\
                           A^TM^2A  \\
                           \vdots  \\
                           A^TM^nA
                         \end{array}
                       \right),$$  we should show that $\mathcal{A}\cong\mathcal{B}$. Define a $\k$-linear
map $f:\mathcal{A}^1\to \mathcal{B}^1$ by $$\left(\begin{array}{c}
                           f(x_1) \\
                           f(x_2) \\
                           \vdots \\
                           f(x_n)
                         \end{array}
                       \right)= A\left(
                         \begin{array}{c}
                           y_1\\
                           y_2 \\
                           \vdots \\
                           y_n
                         \end{array}
                       \right).$$
                       Obviously, $f$ is invertible since $A\in \mathrm{GL}_n(\k)$. Extend $f$  to a morphism of graded algebras between $\mathcal{A}^{\#}$ and $\mathcal{B}^{\#}$. We still denote it by $f$. For any $i\in \{1,2,\cdots,n\}$, we still have (\ref{Eq1}) and (\ref{Eq2}).  Since $$(a_{ij}E_n)_{n^2\times n^2}\left(
                         \begin{array}{c}
                           N^1 \\
                           N^2  \\
                           \vdots  \\
                           N^n
                         \end{array}
                       \right)=\left(
                         \begin{array}{c}
                           A^TM^1A \\
                           A^TM^2A  \\
                           \vdots  \\
                           A^TM^nA
                         \end{array}
                       \right),$$ we have  $f\circ \partial_{\mathcal{A}}(x_i)=\partial_{\mathcal{B}}\circ f(x_i)$, for any $i\in \{1,2,\cdots, n\}$. So $f$ is an isomorphism of DG algebras.
\end{proof}

\begin{cor}\label{judge}
Let $\mathcal{A}$ and $\mathcal{B}$ be two DG free algebras such that $$ \mathcal{A}^{\#}=\k\langle x_1,x_2,\cdots, x_n\rangle, \quad \mathcal{B}^{\#}=\k\langle y_1,y_2,\cdots, y_n\rangle, $$
with each $|x_i|=|y_i|=1$. If $A\cong B$ and $\partial_{\mathcal{A}}$ and $\partial_{\mathcal{B}}$ are defined  respectively by  two crisscross ordered $n$-tuples  $(M^1, M^2,\cdots, M^n)$ and $(N^1,N^2,\cdots, N^n)$ of $n\times n$ matrixes, then
\begin{enumerate}
\item  $r\left(
                         \begin{array}{c}
                           M^1 \\
                           M^2  \\
                           \vdots  \\
                           M^n
                         \end{array}
                       \right)=r\left(
                         \begin{array}{c}
                           N^1 \\
                           N^2  \\
                           \vdots  \\
                           N^n
                         \end{array}
                       \right)$;\\
\item  $N^1, N^2,\cdots, N^n$ are symmetric matrixes whenever $M^1, M^2,\cdots, M^n$ are symmetric matrixes.
\end{enumerate}
\end{cor}
\begin{proof}
(1) By Theorem \ref{isom}, there exists $A=(a_{ij})_{n\times n}\in \mathrm{GL}_n(\k)$ such that
\begin{align}\label{isocond}(a_{ij}E_n)_{n^2\times n^2}\left(
                         \begin{array}{c}
                           N_1 \\
                           N_2  \\
                           \vdots  \\
                           N_n
                         \end{array}
                       \right)=\left(
                         \begin{array}{c}
                           A^TM^1A \\
                           A^TM^2A  \\
                           \vdots  \\
                           A^TM^nA
                         \end{array}
                       \right)
                       \end{align}
                         since $\mathcal{A}\cong \mathcal{B}$.  We have
                       $$\left(
                         \begin{array}{c}
                           A^TM^1A \\
                           A^TM^2A  \\
                           \vdots  \\
                           A^TM^nA
                         \end{array}
                       \right)=\left(
                                 \begin{array}{cccc}
                                   A^T & 0_{n\times n} & \cdots & 0_{n\times n} \\
                                   0_{n\times n} & A^T & \cdots & 0_{n\times n} \\
                                   \vdots & \vdots & \ddots & \vdots  \\
                                   0_{n\times n} & 0_{n\times n} & \cdots & A^T \\
                                 \end{array}
                               \right)\left(
                         \begin{array}{c}
                           M^1 \\
                           M^2  \\
                           \vdots  \\
                           M^n
                         \end{array}
                       \right)A.
                           $$
                       Since $A\in \mathrm{GL}_n(\k)$ and $$\left(
                                 \begin{array}{cccc}
                                   A^T & 0_{n\times n} & \cdots & 0_{n\times n} \\
                                   0_{n\times n} & A^T & \cdots & 0_{n\times n} \\
                                   \vdots & \vdots & \ddots & \vdots  \\
                                   0_{n\times n} & 0_{n\times n} & \cdots & A^T \\
                                 \end{array}
                               \right)\in \mathrm{GL}_{n^2}(\k),$$ we have
                               $$r\left(
                         \begin{array}{c}
                           A^TM^1A \\
                           A^TM^2A  \\
                           \vdots  \\
                           A^TM^nA
                         \end{array}
                       \right)=r\left(
                         \begin{array}{c}
                           M^1 \\
                           M^2  \\
                           \vdots  \\
                           M^n
                         \end{array}
                       \right).$$
On the other hand, let $A^{-1}=(b_{ij})_{n\times n}$,  then
$$(b_{ij}E_n)_{n^2\times n^2}(a_{ij}E_n)_{n^2\times n^2}=E_{n^2},$$ which implies that $(a_{ij}E_n)_{n^2\times n^2}\in \mathrm{GL}_{n^2}(\k)$ and hence
$$r\left(
                         \begin{array}{c}
                           N_1 \\
                           N_2  \\
                           \vdots  \\
                           N_n
                         \end{array}
                       \right)=r\left(
                         \begin{array}{c}
                           A^TM^1A \\
                           A^TM^2A  \\
                           \vdots  \\
                           A^TM^nA
                         \end{array}
                       \right)$$
                       by (\ref{isocond}). Therefore, $r\left(
                         \begin{array}{c}
                           M^1 \\
                           M^2  \\
                           \vdots  \\
                           M^n
                         \end{array}
                       \right)=r\left(
                         \begin{array}{c}
                           N^1 \\
                           N^2  \\
                           \vdots  \\
                           N^n
                         \end{array}
                       \right).$

(2)By (\ref{isocond}), we have \begin{align*}\left(
                         \begin{array}{c}
                           N_1 \\
                           N_2  \\
                           \vdots  \\
                           N_n
                         \end{array}
                       \right)&=(b_{ij}E_n)_{n^2\times n^2}\left(
                         \begin{array}{c}
                           A^TM^1A \\
                           A^TM^2A  \\
                           \vdots  \\
                           A^TM^nA
                         \end{array}
                       \right) \\
                       &=\left(
                         \begin{array}{c}
                          \sum\limits_{j=1}^n b_{1j}A^TM^jA \\
                           \sum\limits_{j=1}^nb_{2j}A^TM^jA  \\
                           \vdots  \\
                          \sum\limits_{j=1}^nb_{nj}A^TM^jA
                         \end{array}
                       \right).
                       \end{align*}
Hence, $N^i=\sum\limits_{j=1}^n b_{ij}A^TM^jA$ and $N^i$ is a symmetric matrix when $M^1,M^2,\cdots, M^n$ are symmetric matrixes, $i=1,2,\cdots, n$.
\end{proof}

\begin{cor}
Let $\mathcal{A}$ be the DG free algebra such that $ \mathcal{A}^{\#}=\k\langle x_1,x_2,\cdots, x_n\rangle$ with $|x_i|=1, i=1,2,\cdots, n$. Assume that $\partial_{\mathcal{A}}$ is determined by a crisscross ordered $n$-tuple $(M^1,M^2,\cdots, M^n)$ of $n\times n$ matrixes. Then
\begin{align*}
\mathrm{Aut}_{dg}(\mathcal{A})=\{A=(a_{ij})_{n\times n}\in \mathrm{GL}_n(\k)|(a_{ij}E_n)_{n^2\times n^2}\left(
                         \begin{array}{c}
                           M^1 \\
                           M^2  \\
                           \vdots  \\
                           M^n
                         \end{array}
                       \right)=\left(
                         \begin{array}{c}
                           A^TM^1A \\
                           A^TM^2A  \\
                           \vdots  \\
                           A^TM^nA
                         \end{array}
                       \right) \}.
\end{align*}
\end{cor}

\section{isomorphism classes of dg free algebras with two generators }
By Remark \ref{onetoone}, the set of crisscross ordered $2$-tuples of $2\times 2$ matrixes
are in one to one correspondence with the set of DG free algebras with two degree one generators.
Hence we should describe all crisscross ordered $2$-tuples of $2\times 2$ matrixes first in order to
figure out the isomorphism classes of DG free algebras with two generators. For this, let $M^1$ and $M^2$ be two $2\times 2$ matrixes such that
\begin{align*}
M^1&=(c_1^1,c_2^1)=\left(
                         \begin{array}{c}
                           r^1_1 \\
                           r^1_2
                         \end{array}
                       \right)=\left(
                                 \begin{array}{cc}
                                   m_{11}^1 & m_{12}^1 \\
                                   m_{21}^1 & m_{22}^1
                                 \end{array}
                               \right),  \\
                        M^2& =(c_1^2,c_2^2)=\left(
                         \begin{array}{c}
                           r^2_1 \\
                           r^2_2
                         \end{array}
                       \right)=\left(
                                 \begin{array}{cc}
                                   m_{11}^2 & m_{12}^2 \\
                                   m_{21}^2 & m_{22}^2
                                 \end{array}
                               \right).
                       \end{align*}
                       By Definition \ref{crisscrossed}, $(M_1,M_2)$ is a  crisscross ordered $2$-tuple of $2\times 2$ matrixes if and only if
                       \begin{align*}
\begin{cases}
c_1^2r_2^1-c_2^1r_1^2=(0)_{2\times 2} & (1)\\
c_2^1r_1^1-c_1^1r_2^1+c_2^2r_2^1-c_2^1r_2^2=(0)_{2\times 2} & (2) \\
c_1^1r_1^2-c_1^2r_1^1+c_1^2r_2^2-c_2^2r_1^2=(0)_{2\times 2} & (3),\\
\end{cases}
\end{align*}
if and only if
\begin{align}\label{twocases}
\begin{cases}
m_{11}^2m_{21}^1-m_{12}^1m_{11}^2=0 & (1.1)\\
m_{11}^2m_{22}^1-m_{12}^1m_{12}^2=0 & (1.2)\\
m_{21}^2m_{21}^1-m_{22}^1m_{11}^2=0 & (1.3)\\
m_{21}^2m_{22}^1-m_{22}^1m_{12}^2=0 & (1.4)\\
m_{12}^1m_{11}^1-m_{11}^1m_{21}^1+m_{12}^2m_{21}^1-m_{12}^1m_{21}^2=0 & (2.1)\\
(m_{12}^1)^2-m_{11}^1m_{22}^1+m_{12}^2m_{22}^1-m_{12}^1m_{22}^2=0 &(2.2) \\
m_{22}^1m_{11}^1-(m_{21}^1)^2+m_{21}^1m_{22}^2-m_{22}^1m_{21}^2=0 &(2.3)\\
m_{22}^1m_{12}^1-m_{21}^1m_{22}^1=0     &(2.4)\\
m_{11}^2m_{21}^2-m_{12}^2m_{11}^2=0  &(3.1)\\
m_{11}^1m_{12}^2-m_{11}^2m_{12}^1+m_{11}^2m_{22}^2-(m_{12}^2)^2=0 &(3.2)\\
m_{21}^1m_{11}^2-m_{21}^2m_{11}^1+(m_{21}^2)^2-m_{22}^2m_{11}^2=0 &(3.3)\\
m_{21}^1m_{12}^2-m_{21}^2m_{12}^1+m_{21}^2m_{22}^2-m_{22}^2m_{12}^2=0 &(3.4),
\end{cases}
\end{align}
where equations $(i.1),(i.2),(i.3),(i.4)$ are equivalent to the equation $(i)$, $i=1,2,3$.
When $m_{11}^2=m_{22}^1=0$, the equations (\ref{twocases}) are equivalent to \begin{align}\label{twozero}
\begin{cases}
0=m_{12}^1m_{12}^2=m_{21}^2m_{21}^1 \\
m_{12}^1(m_{11}^1-m_{21}^2)-m_{21}^1(m_{11}^1-m_{12}^2)=0 \\
m_{12}^1(m_{12}^1-m_{22}^2)=0  \\
m_{21}^1(m_{22}^2-m_{21}^1)=0 \\
m_{12}^2(m_{11}^1-m_{12}^2)=0 \\
m_{21}^2(m_{11}^1-m_{21}^2)=0 \\
m_{12}^2(m_{21}^1-m_{22}^2)-m_{21}^2(m_{12}^1-m_{22}^2)=0.
\end{cases}
\end{align}
By computations, we have the following classification when $m_{22}^1=m_{11}^2=0$.\\
\begin{tabular}{|l|l|l|l|}
  \hline
  Cases for $m_{22}^1=m_{11}^2=0$ & $M^1$ & $M^2$ & Parameters \\  \hline
  $1.\begin{cases}
  m_{12}^1= 0, m_{21}^1= 0 \\
  m_{12}^2\neq 0, m_{21}^2\neq 0
  \end{cases}$ & $\left(
                                       \begin{array}{cc}
                                         \mu &  0 \\
                                         0 & 0\\
                                       \end{array}
                                     \right)$
    & $\left(
        \begin{array}{cc}
          0 & \mu \\
          \mu & \lambda \\
        \end{array}
      \right) $
     & $\mu  \in \k^{\times}, \lambda\in \k$  \\
  \hline
    $2.\begin{cases}
  m_{12}^1= 0, m_{21}^1= 0,\\
  m_{12}^2=0, m_{21}^2\neq 0
  \end{cases}$ & $\left(
                                       \begin{array}{cc}
                                         \nu &  0 \\
                                         0 & 0\\
                                       \end{array}
                                     \right)$
    & $\left(
        \begin{array}{cc}
          0 & 0\\
          \nu & 0 \\
        \end{array}
      \right) $
     & $\nu  \in \k^{\times}$  \\
     \hline
       $3. \begin{cases}
  m_{12}^1= 0, m_{21}^1= 0,\\
  m_{12}^2\neq 0, m_{21}^2= 0
  \end{cases}$ & $\left(
                                       \begin{array}{cc}
                                         \nu &  0 \\
                                         0 & 0\\
                                       \end{array}
                                     \right)$
    & $\left(
        \begin{array}{cc}
          0 & \nu \\
          0 & 0 \\
        \end{array}
      \right) $
     & $\nu  \in \k^{\times}$  \\
     \hline
      $4.\begin{cases}
  m_{12}^1= 0, m_{21}^1= 0 \\
  m_{12}^2 = 0, m_{21}^2 = 0
  \end{cases}$ & $\left(
                                       \begin{array}{cc}
                                         \lambda &  0 \\
                                         0 & 0\\
                                       \end{array}
                                     \right)$
    & $\left(
        \begin{array}{cc}
          0 & 0 \\
          0 & \mu \\
        \end{array}
      \right) $
     & $\lambda, \mu \in \k$  \\
   \hline
    $5.\begin{cases}
  m_{12}^1= 0, m_{21}^2= 0 \\
  m_{21}^1\neq 0
  \end{cases}$ & $\left(
                                       \begin{array}{cc}
                                         \nu &  0 \\
                                         \mu & 0\\
                                       \end{array}
                                     \right)$
    & $\left(
        \begin{array}{cc}
          0 & \nu \\
          0 & \mu \\
        \end{array}
      \right) $
     & $\mu  \in \k^{\times}, \nu \in \k$  \\
     \hline
      $6.\begin{cases}
  m_{12}^1= 0, m_{21}^2= 0 \\
  m_{12}^2\neq 0
  \end{cases}$ & $\left(
                                       \begin{array}{cc}
                                         \nu &  0 \\
                                         \mu & 0\\
                                       \end{array}
                                     \right)$
    & $\left(
        \begin{array}{cc}
          0 & \nu \\
          0 & \mu \\
        \end{array}
      \right) $
     & $\mu  \in \k, \nu \in \k^{\times}$  \\
     \hline
   $7.\begin{cases}
  m_{12}^2= 0, m_{21}^2= 0 \\
  m_{11}^1\neq 0,m_{12}^1\neq 0
  \end{cases}$ & $\left(
                                       \begin{array}{cc}
                                         \lambda &  \mu \\
                                         \mu & 0\\
                                       \end{array}
                                     \right)$
    & $\left(
        \begin{array}{cc}
          0 & 0 \\
          0 & \mu \\
        \end{array}
      \right) $
     & $\lambda, \mu \in \k^{\times}$  \\
     \hline
      $8.\begin{cases}
  m_{12}^2= 0, m_{21}^2= 0 \\
  m_{11}^1=0, m_{12}^1\neq 0, m_{21}^1=0
  \end{cases}$ & $\left(
                                       \begin{array}{cc}
                                         0 &  \mu \\
                                         0 & 0\\
                                       \end{array}
                                     \right)$
    & $\left(
        \begin{array}{cc}
          0 & 0 \\
          0 & \mu \\
        \end{array}
      \right) $
     & $\mu \in \k^{\times}$  \\
     \hline
      $9.\begin{cases}
  m_{12}^2= 0, m_{21}^2= 0 \\
  m_{11}^1=0, m_{12}^1\neq 0, m_{21}^1\neq 0
  \end{cases}$ & $\left(
                                       \begin{array}{cc}
                                         0 &  \nu \\
                                         \nu & 0\\
                                       \end{array}
                                     \right)$
    & $\left(
        \begin{array}{cc}
          0 & 0 \\
          0 & \nu \\
        \end{array}
      \right) $
     & $\nu \in \k^{\times}$  \\
     \hline
     $10.\begin{cases}
  m_{12}^2= 0, m_{21}^2= 0 \\
  m_{11}^1=0, m_{12}^1= 0, m_{21}^1\neq 0
  \end{cases}$ & $\left(
                                       \begin{array}{cc}
                                         0 &  0 \\
                                         \mu & 0\\
                                       \end{array}
                                     \right)$
    & $\left(
        \begin{array}{cc}
          0 & 0 \\
          0 & \mu \\
        \end{array}
      \right) $
     & $\mu \in \k^{\times}$  \\
     \hline
     $11.\begin{cases}
  m_{12}^2= 0, m_{21}^1= 0 \\
  m_{12}^1\neq 0
  \end{cases}$ & $\left(
                                       \begin{array}{cc}
                                         \nu &  \mu \\
                                         0 & 0\\
                                       \end{array}
                                     \right)$
    & $\left(
        \begin{array}{cc}
          0 & 0 \\
          \nu & \mu \\
        \end{array}
      \right) $
     & $\mu \in \k^{\times}, \nu\in \k$  \\
     \hline
      $12. \begin{cases}
  m_{12}^2= 0, m_{21}^1= 0 \\
  m_{21}^2\neq0
  \end{cases}$ & $\left(
                                       \begin{array}{cc}
                                         \nu &  \mu \\
                                         0 & 0\\
                                       \end{array}
                                     \right)$
    & $\left(
        \begin{array}{cc}
          0 & 0 \\
          \nu & \mu \\
        \end{array}
      \right) $
     & $\mu \in \k, \nu\in \k^{\times}$  \\
     \hline
  \end{tabular}

When $(m_{22}^1,m_{11}^2)\neq (0,0)$, we have the
 following classification by computations.
\begin{tabular}{|l|l|l|l|}
  \hline
  Cases & $M^1$ & $M^2$ & Parameters \\  \hline
  $13. \begin{cases}
  m_{11}^2\neq 0, \\
  m_{22}^1\neq 0
  \end{cases}$ & $\left(
                                       \begin{array}{cc}
                                         \nu+\lambda(\mu-\omega) & \mu \\
                                         \mu & \frac{\mu}{\lambda}\\
                                       \end{array}
                                     \right)$
    & $\left(
        \begin{array}{cc}
          \lambda\nu & \nu \\
          \nu & \omega \\
        \end{array}
      \right) $
     & $\mu, \nu,\lambda \in \k^{\times}, \omega\in \k$  \\
  \hline
  $14.\begin{cases}m_{11}^2=0,\\
  m_{22}^1\neq 0,\\
  m_{12}^2=0
  \end{cases}$ & $\left(
        \begin{array}{cc}
          \frac{\mu(\mu-\omega)}{\lambda} & \mu \\
          \mu & \lambda \\
        \end{array}
      \right) $& $\left(
                   \begin{array}{cc}
                     0 & 0 \\
                     0 & \omega \\
                   \end{array}
                 \right)$
       & $\lambda\in \k^{\times}, \mu,\omega \in \k$ \\
   \hline
   $15.\begin{cases}m_{11}^2=0,\\
  m_{22}^1\neq 0,\\
  m_{12}^1=0
  \end{cases}$ & $\left(
        \begin{array}{cc}
          \omega &  0\\
          0 & \lambda \\
        \end{array}
      \right) $& $\left(
                   \begin{array}{cc}
                     0 & \omega \\
                     \omega & \mu \\
                   \end{array}
                 \right)$
       & $\lambda\in \k^{\times}, \mu,\omega \in \k$ \\
   \hline
   $16. \begin{cases}m_{11}^2\neq 0,\\
  m_{22}^1= 0,\\
  m_{12}^1=0
  \end{cases}$ & $\left(
        \begin{array}{cc}
        \omega & 0 \\
          0 & 0 \\
        \end{array}
      \right) $& $\left(
                   \begin{array}{cc}
                     \lambda & \mu \\
                     \mu & \frac{\mu(\mu-\omega)}{\lambda} \\
                   \end{array}
                 \right)$
       & $\lambda\in \k^{\times}, \mu,\omega \in \k$ \\
   \hline
   $17. \begin{cases}m_{11}^2\neq 0,\\
  m_{22}^1= 0,\\
  m_{12}^2=0
  \end{cases}$ & $\left(
        \begin{array}{cc}
          \mu & \omega \\
          \omega & 0 \\
        \end{array}
      \right) $& $\left(
                   \begin{array}{cc}
                     \lambda & 0 \\
                     0 & \omega\\
                   \end{array}
                 \right)$
       & $\lambda\in \k^{\times}, \mu,\omega \in \k$ \\
   \hline
\end{tabular}
Now, lets come back to our concerned isomorphism problem.
Let $\mathcal{A}$ and $\mathcal{B}$ be two DG free algebras such that $$ \mathcal{A}^{\#}=\k\langle x_1,x_2\rangle, \quad \mathcal{B}^{\#}=\k\langle y_1,y_2 \rangle, $$
with each $|x_i|=|y_i|=1$. Assume that $\partial_{\mathcal{A}}$ and $\partial_{\mathcal{B}}$ are defined by crisscross ordered $2$-tuples $(M^1,M^2)$ and $(N^1,N^2)$ of $2\times 2$ matrixes, respectively. Let
\begin{align*}
M^i= \left(
       \begin{array}{cc}
         m^i_{11} & m^i_{12} \\
         m^i_{21} & m^i_{22} \\
       \end{array}
     \right),
    N^i= \left(
       \begin{array}{cc}
         n^i_{11} & n^i_{12} \\
         n^i_{21} & n^i_{22} \\
       \end{array}
     \right), i=1,2.
\end{align*}
 By Theorem \ref{isom}, $\mathcal{A}\cong \mathcal{B}$ if and only if there exists $A=(a_{ij})_{2\times 2}\in \mathrm{GL}_2(\k)$ such that
\begin{align}\label{eqs}
(a_{ij}E_2)_{4\times 4}\left(
                         \begin{array}{c}
                           N^1 \\
                           N^2
                         \end{array}
                       \right)=\left(
                         \begin{array}{c}
                           A^TM^1A \\
                           A^TM^2A
                         \end{array}
                       \right),
                       \end{align}
i.e.,
\begin{align*}
\begin{cases}
a_{11}n_{11}^1+a_{12}n_{11}^2 = a_{11}^2m_{11}^1+a_{11}a_{21}m_{21}^1+a_{11}a_{21}m_{12}^1+a_{21}^2m_{22}^1 \\
a_{11}n_{12}^1+a_{12}n_{12}^2 = a_{11}a_{12}m_{11}^1+a_{12}a_{21}m_{21}^1+a_{11}a_{22}m_{12}^1+a_{21}a_{22}m_{22}^1\\
 a_{11}n_{21}^1+a_{12}n_{21}^2 = a_{12}a_{11}m_{11}^1+a_{11}a_{22}m_{21}^1+a_{12}a_{21}m_{12}^1+a_{21}a_{22}m_{22}^1\\
 a_{11}n_{22}^1+a_{12}n_{22}^2 = a_{12}^2m_{11}^1+a_{12}a_{22}m_{21}^1+a_{12}a_{22}m_{12}^1+a_{22}^2m_{22}^1  \\
  a_{21}n_{11}^1+a_{22}n_{11}^2 =  a_{11}^2m_{11}^2+a_{11}a_{21}m_{21}^2+a_{11}a_{21}m_{12}^2+a_{21}^2m_{22}^2 \\
  a_{21}n_{12}^1+a_{22}n_{12}^2 = a_{11}a_{12}m_{11}^2+a_{12}a_{21}m_{21}^2+a_{11}a_{22}m_{12}^2+a_{21}a_{22}m_{22}^2\\
  a_{21}n_{21}^1+a_{22}n_{21}^2 = a_{11}a_{12}m_{11}^2+a_{11}a_{22}m_{21}^2+a_{12}a_{21}m_{12}^2+a_{21}a_{22}m_{22}^2 \\
  a_{21}n_{22}^1+a_{22}n_{22}^2 =  a_{12}^2m_{11}^2+a_{12}a_{22}m_{21}^2+a_{12}a_{22}m_{12}^2+a_{22}^2m_{22}^2.
\end{cases}
\end{align*}
If $\mathcal{A}\cong \mathcal{B}$, then by Corollary \ref{judge}, we have

$(1) r\left(
                         \begin{array}{c}
                           N^1 \\
                           N^2
                         \end{array}
                       \right)=r\left(
                         \begin{array}{c}
                           M^1 \\
                           M^2
                         \end{array}
                       \right)$;

$(2) N^1, N^2$ are symmetric matrixes whenever $M^1, M^2$ are symmetric matrixes.

\begin{prop}\label{nonsym}
Assume that $\mathcal{A}$ is a DG free algebra such that $$\mathcal{A}^{\#}=\k\langle x_1,x_2\rangle, |x_1|=|x_2|=1$$ and $\partial_{\mathcal{A}}$ is defined by a crisscross ordered $2$-tuple $(M^1, M^2)$ of $2\times 2$ matrixes. If $M^1$ and $M^2$ are not both symmetric matrixes, then $\mathcal{A}$ is isomorphic to either of the following two DG free algebras:
\begin{enumerate}
\item $\mathcal{B}_1$ whose differential $\partial_{\mathcal{B}_1}$ is defined by $ \left(
       \begin{array}{cc}
         1 & 0\\
         0 & 0 \\
       \end{array}
     \right)$ and $  \left(
       \begin{array}{cc}
         0 & 0 \\
         1 & 0\\
       \end{array}
     \right);  $ \\
\item  $\mathcal{B}_2$ whose differential $\partial_{\mathcal{B}_2}$ is defined by $ \left(
       \begin{array}{cc}
         1 & 0\\
         0 & 0 \\
       \end{array}
     \right)$ and $  \left(
       \begin{array}{cc}
         0 & 1 \\
         0 & 0\\
       \end{array}
     \right)$.
\end{enumerate}

\end{prop}
\begin{proof}
Since $M^1$ and $M^2$ are not both symmetric matrixes, we have $m_{11}^2=m_{22}^1=0$,
and the crisscross ordered $2$-tuple $(M^1,M^2)$ of $2\times 2$ matrixes belongs to one of following cases:
\begin{align*}
\text{Case}\quad\quad  &\quad\quad M^1 & M^2\quad\quad &\quad\quad \text{Parameters}  \\
2.\quad\quad &\left(
                                       \begin{array}{cc}
                                         \nu &  0 \\
                                         0 & 0\\
                                       \end{array}
                                     \right) & \left(
        \begin{array}{cc}
          0 & 0\\
          \nu & 0 \\
        \end{array}
      \right)&\quad\quad \nu\in \k^{\times}  \\
3.\quad\quad &\left(
                                       \begin{array}{cc}
                                         \nu &  0 \\
                                         0 & 0\\
                                       \end{array}
                                     \right) & \left(
        \begin{array}{cc}
          0 & \nu \\
          0 & 0 \\
        \end{array}
      \right) &\quad\quad \nu  \in \k^{\times}    \\
5. \quad\quad & \left(
                                       \begin{array}{cc}
                                         \nu &  0 \\
                                         \mu & 0\\
                                       \end{array}
                                     \right)   & \left(
        \begin{array}{cc}
          0 & \nu \\
          0 & \mu \\
        \end{array}
      \right)             & \quad\quad \mu  \in \k^{\times}, \nu \in \k    \\
6.\quad\quad &\left(
                                       \begin{array}{cc}
                                         \nu &  0 \\
                                         \mu & 0\\
                                       \end{array}
                                     \right) &\left(
        \begin{array}{cc}
          0 & \nu \\
          0 & \mu \\
        \end{array}
      \right) & \quad\quad \mu  \in \k, \nu \in \k^{\times}    \\
8.\quad\quad &\left(
                                       \begin{array}{cc}
                                         0 &  \mu \\
                                         0 & 0\\
                                       \end{array}
                                     \right) &\left(
        \begin{array}{cc}
          0 & 0 \\
          0 & \mu \\
        \end{array}
      \right) & \quad\quad \mu \in \k^{\times}    \\
10.\quad\quad &\left(
                                       \begin{array}{cc}
                                         0 &  0 \\
                                         \mu & 0\\
                                       \end{array}
                                     \right) & \left(
        \begin{array}{cc}
          0 & 0 \\
          0 & \mu \\
        \end{array}
      \right)&\quad\quad \mu \in \k^{\times}     \\
11.\quad\quad &\left(
                                       \begin{array}{cc}
                                         \nu &  \mu \\
                                         0 & 0\\
                                       \end{array}
                                     \right) &\left(
        \begin{array}{cc}
          0 & 0 \\
          \nu & \mu \\
        \end{array}
      \right) & \quad\quad \mu \in \k^{\times}, \nu\in \k    \\
12.\quad\quad &\left(
                                       \begin{array}{cc}
                                         \nu &  \mu \\
                                         0 & 0\\
                                       \end{array}
                                     \right) &\left(
        \begin{array}{cc}
          0 & 0 \\
          \nu & \mu \\
        \end{array}
      \right) &\quad\quad \mu \in \k, \nu\in \k^{\times}   \\
\end{align*}
by the classification above. Let $(M^1,M^2)$ and $(N^1,N^2)$ are two crisscross ordered $2$-tuples of $2\times 2$ matrixes belong to one of the $8$ cases listed above. We want to check whether there exists $A\in \mathrm{GL}_2(\k)$ such that $(\ref{eqs})$ holds. By computations,
\begin{align*}
 & \quad (M^1, M^2)& (N^1,N^2)  &  \quad\quad \exists A   \\
\text{Cases}& \quad \quad 2 &\quad  8   \quad    &\quad \left(
        \begin{array}{cc}
          0 & \frac{\mu}{\nu} \\
          1 & 0 \\
        \end{array}
      \right) \\
\text{Cases} & \quad \quad 8 &\quad  11  \quad    &\quad \left(
        \begin{array}{cc}
          1 & 0 \\
          \frac{\nu}{\mu} & 1 \\
        \end{array}
      \right) \\
\text{Cases} & \quad \quad 3 &\quad  6  \quad    &\quad \left(
        \begin{array}{cc}
          1 & \frac{\mu}{\nu} \\
          0 & 1 \\
        \end{array}
      \right)  \\
\text{Cases} & \quad \quad 10 &\quad  5  \quad    &\quad \left(
        \begin{array}{cc}
          1+\frac{\nu^2}{\mu^2} & \frac{\nu}{\mu} \\
          \frac{\nu}{\mu} & 1 \\
        \end{array}
      \right)\\
\text{Cases} & \quad \quad 3 &\quad  10  \quad    &\quad \left(
        \begin{array}{cc}
          0 & \frac{\mu}{\nu} \\
          1 & 0 \\
        \end{array}
      \right).
\end{align*}
And it is easy to see that the crisscross ordered $2$-tuple $(M^1, M^2)$ of $2\times 2$ matrixes in  Case $12$ belongs to Case $11$ and Case $2$ when $\mu\in \k^{\times}$ and $\mu =0$, respectively.

 By Theorem \ref{isom}, $\mathcal{A}\cong \mathcal{B}_1$ when the crisscross ordered $2$-tuple $(M^1, M^2)$ of $2\times 2$ matrixes belongs to any one of the following cases:
 $ \text{Case}\, 2,  \text{Case}\,8, \text{Case}\,11,$ and  $\text{Case}\,12$. And $\mathcal{A}\cong \mathcal{B}_2$  when  $(M^1, M^2)$ belongs to Case $3$, Case $5$, Case $6$ and Case $10$. Since
 $$1=r\left(
      \begin{array}{cc}
        1 & 0 \\
        0 & 0 \\
        0 & 0 \\
        1 & 0 \\
      \end{array}
    \right)\neq
   r\left(
      \begin{array}{cc}
        1 & 0 \\
        0 & 0 \\
        0 & 1 \\
        0 & 0 \\
      \end{array}
    \right)=2,
  $$
we have $\mathcal{B}_1\not\cong \mathcal{B}_2$ by Corollary \ref{judge}.

\end{proof}

  \begin{prop}\label{symone}
Assume that $\mathcal{A}$ is a DG free algebra such that $$\mathcal{A}^{\#}=\k\langle x_1,x_2\rangle, |x_1|=|x_2|=1$$ and $\partial_{\mathcal{A}}$ is defined by crisscrossed $2\times 2$ matrixes $M^1$ and $M^2$. If $M^1$ and $M^2$ are both symmetric matrixes with $m_{22}^1=m_{11}^2=0$, then $\mathcal{A}$ is isomorphic to one of the  following four DG free algebras:
\begin{enumerate}
\item $\mathcal{B}_3$ whose differential $\partial_{\mathcal{B}_3}$ is defined by $ \left(
       \begin{array}{cc}
         1 & 0\\
         0 & 0 \\
       \end{array}
     \right)$ and $  \left(
       \begin{array}{cc}
         0 & 1 \\
         1 & 0\\
       \end{array}
     \right);  $ \\
\item  $\mathcal{B}_4$ whose differential $\partial_{\mathcal{B}_4}$ is defined by $ \left(
       \begin{array}{cc}
         1 & 0\\
         0 & 0 \\
       \end{array}
     \right)$ and $  \left(
       \begin{array}{cc}
         0 & 0 \\
         0 & 1\\
       \end{array}
     \right)$;
   \item $\mathcal{B}_5$ whose differential $\partial_{\mathcal{B}_5}$ is defined by $ \left(
       \begin{array}{cc}
         1 & 0\\
         0 & 0 \\
       \end{array}
     \right)$ and $  \left(
       \begin{array}{cc}
         0 & 0 \\
         0 & 0\\
       \end{array}
     \right);  $ \\
     \item $\mathcal{B}_0$ whose differential $\partial_{\mathcal{B}_0}=0$.
\end{enumerate}
  \end{prop}
  \begin{proof}
  Since $M^1$ and $M^2$ are both symmetric matrixes  with $m_{22}^1=m_{11}^2=0$,
 the crisscrossed $2\times 2$ matrixes pair $(M^1, M^2)$ belongs to one of following cases:
\begin{align*}
\text{Case}\quad\quad  &\quad\quad M^1 & M^2\quad\quad &\quad\quad \text{Parameters}  \\
1.\quad\quad &\left(
                                       \begin{array}{cc}
                                         \mu &  0 \\
                                         0 & 0\\
                                       \end{array}
                                     \right) & \left(
        \begin{array}{cc}
          0 & \mu\\
          \mu & \lambda \\
        \end{array}
      \right)&\quad \quad\quad \mu\in \k^{\times}, \lambda\in\k  \\
4.\quad\quad &\left(
                                       \begin{array}{cc}
                                         \lambda &  0 \\
                                         0 & 0\\
                                       \end{array}
                                     \right) & \left(
        \begin{array}{cc}
          0 & 0\\
          0 & \mu \\
        \end{array}
      \right)&\quad \quad\quad \lambda, \mu\in \k \\
7.\quad\quad &\left(
                                       \begin{array}{cc}
                                         \lambda &  \mu \\
                                         \mu & 0\\
                                       \end{array}
                                     \right) &\left(
        \begin{array}{cc}
          0 & 0 \\
          0 & \mu \\
        \end{array}
      \right) & \quad \quad \quad\lambda, \mu \in \k^{\times}    \\
9.\quad\quad &\left(
                                       \begin{array}{cc}
                                         0 &  \nu \\
                                         \nu & 0\\
                                       \end{array}
                                     \right) & \left(
        \begin{array}{cc}
          0 & 0 \\
          0 & \nu \\
        \end{array}
      \right)&\quad\quad\quad \nu \in \k^{\times}     \\
\end{align*}
by the classification above.  By computations,
for \begin{align*}
M^1=\left(
                                       \begin{array}{cc}
                                         \mu &  0 \\
                                         0 & 0\\
                                       \end{array}
                                     \right), & M^2=  \left(
        \begin{array}{cc}
          0 & \mu\\
          \mu & \lambda \\
        \end{array}
      \right) \\
N^1=\left(
                                       \begin{array}{cc}
                                         \lambda &  \mu \\
                                         \mu & 0\\
                                       \end{array}
                                     \right),  & N^2=\left(
        \begin{array}{cc}
          0 & 0 \\
          0 & \mu \\
        \end{array}
      \right),  \lambda,\mu\in \k^{\times}
\end{align*}
and
\begin{align*}
M^1=\left(
                                       \begin{array}{cc}
                                         \lambda &  0 \\
                                         0 & 0\\
                                       \end{array}
                                     \right), & M^2=  \left(
        \begin{array}{cc}
          0 & 0\\
          0 & \mu \\
        \end{array}
      \right) \\
N^1=\left(
                                       \begin{array}{cc}
                                         \lambda &  \mu \\
                                         \mu & 0\\
                                       \end{array}
                                     \right),  & N^2=\left(
        \begin{array}{cc}
          0 & 0 \\
          0 & \mu \\
        \end{array}
      \right), \lambda,\mu\in \k^{\times},
\end{align*}
there exists $A= \left(
        \begin{array}{cc}
          0 & 1 \\
          1 & 0 \\
        \end{array}
      \right)$ and  $A= \left(
        \begin{array}{cc}
          0 & \frac{\mu}{\lambda} \\
          \frac{\lambda}{\mu} & 1 \\
        \end{array}
      \right)$, respectively, such that $(\ref{eqs})$ holds.
On the other hand, for any $\lambda, \mu\in \k^{\times}$, and
\begin{align*}
M^1=\left(
                                       \begin{array}{cc}
                                         \lambda &  0 \\
                                         0 & 0\\
                                       \end{array}
                                     \right), & M^2=  \left(
        \begin{array}{cc}
          0 & 0\\
          0 & \mu \\
        \end{array}
      \right) \\
N^1=\left(
                                       \begin{array}{cc}
                                         1 &  0 \\
                                         0 & 0\\
                                       \end{array}
                                     \right),  & N^2=\left(
        \begin{array}{cc}
          0 & 0 \\
          0 & 1 \\
        \end{array}
      \right),
\end{align*}
there exists $A=\left(
        \begin{array}{cc}
          \frac{1}{\lambda} & 0 \\
          0 & \frac{1}{\mu}\\
        \end{array}
      \right) $ such that $(\ref{eqs})$ holds. Therefore, $\mathcal{A}$ is isomorphic to $\mathcal{B}_4$, when $M^1$ and $M^2$ belong to each of the following cases:
      \begin{enumerate}
      \item Case $7$;
      \item Case $1$, $\lambda\neq 0$;
      \item Case $4$, $\lambda \neq 0, \mu\neq 0$.
      \end{enumerate}
      If the parameter $\lambda=0$ in Case $1$,  then $$M^1= \left(
                                       \begin{array}{cc}
                                         \mu &  0 \\
                                         0 & 0\\
                                       \end{array}
                                     \right), M^2= \left(
        \begin{array}{cc}
          0 & \mu\\
          \mu & 0\\
        \end{array}
      \right),\mu\in\k^{\times}.$$
Let $N^1, N^2$ belong to Case $9$, i.e.,
$$N^1=\left(
                                       \begin{array}{cc}
                                         0 &  \nu \\
                                         \nu & 0\\
                                       \end{array}
                                     \right) , N^2=\left(
        \begin{array}{cc}
          0 & 0 \\
          0 & \nu \\
        \end{array}
      \right).$$
Then there exists $A=\left(
        \begin{array}{cc}
          0 & \frac{\nu}{\mu} \\
          1 & 0 \\
        \end{array}
      \right)$ such that $(\ref{eqs})$ holds. Hence $\mathcal{A}$ is isomorphic to $\mathcal{B}_3$, when
 $M^1$ and $M^2$ belong to either of the following two cases:
 \begin{enumerate}
 \item Case $1$, $\lambda =0$;
 \item Case $9$.
 \end{enumerate}
If the parameter $\lambda\neq 0$ and $\mu=0$ in Case $4$,  then $$M^1= \left(
                                       \begin{array}{cc}
                                         \lambda &  0 \\
                                         0 & 0\\
                                       \end{array}
                                     \right), M^2= \left(
        \begin{array}{cc}
          0 &  0\\
          0 & 0\\
        \end{array}
      \right).$$
 Let $N^1, N^2$ belong to Case $4$ with $\lambda=0, \mu\neq 0$, i.e.,
$$N^1=\left(
                                       \begin{array}{cc}
                                         0 &  0 \\
                                         0 &  0\\
                                       \end{array}
                                     \right) , N^2=\left(
        \begin{array}{cc}
          0 & 0 \\
          0 & \mu \\
        \end{array}
      \right).$$
Then there exists $A=\left(
        \begin{array}{cc}
          0 & \frac{\mu}{\lambda} \\
          1 & 0 \\
        \end{array}
      \right)$ such that $(\ref{eqs})$ holds. So $\mathcal{A}$ is isomorphic to $\mathcal{B}_5$, when
$M^1$ and $M^2$ belong to either of the following two cases:
\begin{enumerate}
\item Case $4$, $\lambda \neq 0,\mu=0$;
\item Case $5$, $\lambda =0, \mu\neq 0$.
\end{enumerate}

\end{proof}
\begin{rem}
We have $\mathcal{B}_3\not\cong \mathcal{B}_4$ since one can't find $A=(a_{ij})_{2\times 2}\in \mathrm{GL}_{2}(\k)$ satisfying $(\ref{eqs})$,i.e.,
$$ \left(
     \begin{array}{cc}
       a_{11} & a_{12}\\
       a_{12}& 0 \\
       a_{21} & a_{22}\\
       a_{22} & 0  \\
     \end{array}
   \right)\\
   =\left(
             \begin{array}{cc}
               a_{11}^2& a_{11}a_{12} \\
               a_{12}a_{11} & a_{12}^2 \\
               a_{21}^2 & a_{21}a_{22} \\
               a_{21}a_{22} & a_{22}^2 \\
             \end{array}
           \right).  $$
   Hence $\mathcal{B}_3, \mathcal{B}_4,\mathcal{B}_5$ and $\mathcal{B}_6$ are $4$ different isomorphism classes by Corollary \ref{judge}.
\end{rem}

\begin{prop}\label{symtwo}
Assume that $\mathcal{A}$ is a DG free algebra such that $$\mathcal{A}^{\#}=\k\langle x_1,x_2\rangle, |x_1|=|x_2|=1$$ and $\partial_{\mathcal{A}}$ is defined by a crisscross ordered $2$-tuple $(M^1,M^2)$ of $2\times 2$ matrixes. If $M^1$ and $M^2$ are both symmetric matrixes with $(m_{22}^1,m_{11}^2)\neq (0,0)$, then $\mathcal{A}$ is isomorphic to one of the  following DG free algebras:
\begin{enumerate}
\item  $\mathcal{B}_5$,  where $\partial_{\mathcal{B}_5}$ is defined by $ \left(
       \begin{array}{cc}
         1& 0\\
         0 & 0 \\
       \end{array}
     \right)$ and $  \left(
       \begin{array}{cc}
         0 & 0 \\
         0 & 0\\
       \end{array}
     \right)$;
\item  $\mathcal{B}_6$, where $\partial_{\mathcal{B}_6}$ is defined by $ \left(
       \begin{array}{cc}
         0& 0\\
         0 & 1 \\
       \end{array}
     \right)$ and $  \left(
       \begin{array}{cc}
         0 & 0 \\
         0 & 0\\
       \end{array}
     \right)$;
     \item  $\mathcal{B}_7$, where $\partial_{\mathcal{B}_7}$ is defined by $ \left(
       \begin{array}{cc}
         0& 1\\
         1 & 1\\
       \end{array}
     \right)$ and $  \left(
       \begin{array}{cc}
         0 & 0 \\
         0 & 1\\
       \end{array}
     \right)$;
\item $\mathcal{B}_8$,  where $\partial_{\mathcal{B}_8}$ is defined by $ \left(
       \begin{array}{cc}
         0 & 0\\
         0 & 1 \\
       \end{array}
     \right)$ and $  \left(
       \begin{array}{cc}
         0 & 0 \\
         0 & 1\\
       \end{array}
     \right);  $ \\
   \item $\mathcal{B}_{9}$, where $\partial_{\mathcal{B}_{9}}$ is defined by $ \left(
       \begin{array}{cc}
         1 & 1\\
         1 & 0 \\
       \end{array}
     \right)$ and $  \left(
       \begin{array}{cc}
         0 & 0 \\
         0 & 1\\
       \end{array}
     \right);  $ \\
     \item $\mathcal{B}_{10}$, where $\partial_{\mathcal{B}_{10}}$ is defined by $ \left(
       \begin{array}{cc}
         1 & 0\\
         0 & -\frac{1}{4} \\
       \end{array}
     \right)$ and $  \left(
       \begin{array}{cc}
         0 & 1 \\
         1 & 1\\
       \end{array}
     \right);  $ \\
      \item $\mathcal{B}_{11}$, where $\partial_{\mathcal{B}_{11}}$ is defined by $ \left(
       \begin{array}{cc}
         1 & 0\\
         0 & 1 \\
       \end{array}
     \right)$ and $  \left(
       \begin{array}{cc}
         0 & 1 \\
         1 & 0\\
       \end{array}
     \right) $;
     \item $\mathcal{B}(s,t)$, where $\partial_{\mathcal{B}(s,t)}$ is defined by $ \left(
       \begin{array}{cc}
         1+s-st & 1\\
         1 & \frac{1}{s}\\
       \end{array}
     \right)$ and $ \left(
       \begin{array}{cc}
         s & 1 \\
         1 & t\\
       \end{array}
     \right), s\in \k^{\times}, t\in \k.$
\end{enumerate}
\end{prop}

\begin{proof}
Since $M^1$ and $M^2$ are both symmetric matrixes  with $(m_{22}^1,m_{11}^2)\neq (0,0)$,
 the crisscross ordered $2$-tuple $(M^1, M^2)$ of $2\times 2$ matrixes  belongs to one of the following cases:
\begin{align*}
\text{Case}\quad\quad  &\quad\quad M^1 & M^2\quad\quad &\quad\quad \text{Parameters}  \\
13.\quad\quad &\left(
                                       \begin{array}{cc}
                                         \nu+\lambda(\mu-\omega) &  \mu \\
                                         \mu & \frac{\mu}{\lambda}\\
                                       \end{array}
                                     \right) & \left(
        \begin{array}{cc}
          \lambda\nu & \nu\\
          \nu & \omega \\
        \end{array}
      \right)&\quad \quad\quad \mu,\nu,\lambda \in \k^{\times}, \omega\in\k  \\
14.\quad\quad &\left(
                                       \begin{array}{cc}
                                         \frac{\mu(\mu-\omega)}{\lambda} &  \mu \\
                                         \mu & \lambda\\
                                       \end{array}
                                     \right) & \left(
        \begin{array}{cc}
          0 & 0\\
          0 & \omega \\
        \end{array}
      \right)&\quad \quad\quad \lambda\in \k^{\times},\mu, \omega\in \k \\
15.\quad\quad &\left(
                                       \begin{array}{cc}
                                         \omega &  0 \\
                                         0 & \lambda\\
                                       \end{array}
                                     \right) &\left(
        \begin{array}{cc}
          0 & \omega \\
          \omega & \mu \\
        \end{array}
      \right) & \quad \quad \quad\lambda\in \k^{\times},\mu,\omega\in \k    \\
16.\quad\quad &\left(
                                       \begin{array}{cc}
                                         \omega &  0 \\
                                         0 & 0\\
                                       \end{array}
                                     \right) & \left(
        \begin{array}{cc}
          \lambda & \mu \\
          \mu & \frac{\mu(\mu-\omega)}{\lambda} \\
        \end{array}
      \right)&\quad\quad\quad \lambda \in \k^{\times},\mu,\omega\in \k     \\
      17.\quad\quad &\left(
                                       \begin{array}{cc}
                                         \mu &  \omega \\
                                         \omega & 0\\
                                       \end{array}
                                     \right) & \left(
        \begin{array}{cc}
          \lambda & 0 \\
          0 & \omega \\
        \end{array}
      \right)&\quad\quad\quad \lambda \in \k^{\times},\mu,\omega\in \k    \\
\end{align*}
by the classification above.
 By computations,
for \begin{align*}
M^1=\left(
                                       \begin{array}{cc}
                                         \frac{\mu(\mu-\omega)}{\lambda} &  \mu \\
                                         \mu & \lambda\\
                                       \end{array}
                                       \right), & M^2=  \left(
        \begin{array}{cc}
          0 & 0\\
          0 & \omega \\
        \end{array}
      \right) \\
N^1=\left(
                                       \begin{array}{cc}
                                         \omega &  0 \\
                                         0 & 0\\
                                       \end{array}
                                     \right),  & N^2=\left(
         \begin{array}{cc}
          \lambda & \mu \\
          \mu & \frac{\mu(\mu-\omega)}{\lambda} \\
        \end{array}
      \right),  \lambda\in \k^{\times},\mu,\omega\in \k
\end{align*}
and
\begin{align*}
M^1=\left(
                                      \begin{array}{cc}
                                         \omega &  0 \\
                                         0 & \lambda\\
                                       \end{array}
                                     \right), & M^2=  \left(
        \begin{array}{cc}
          0 & \omega \\
          \omega & \mu \\
        \end{array}
      \right) \\
N^1=\left(
                                      \begin{array}{cc}
                                         \mu &  \omega \\
                                         \omega & 0\\
                                       \end{array}
                                     \right),  & N^2=\left(
        \begin{array}{cc}
          \lambda & 0 \\
          0 & \omega \\
        \end{array}
      \right), \lambda\in \k^{\times}, \mu,\omega\in \k,
\end{align*}
there exists $A= \left(
        \begin{array}{cc}
          0 & 1 \\
          1 & 0 \\
        \end{array}
      \right)$ such that $(\ref{eqs})$ holds. By Theorem \ref{isom}, we only need to check the isomorphism classes of Case $13$, Case $14$ and Case $15$ one by one.

For Case $14$, we divide it into the following $5$ cases:
\begin{center}
\begin{tabular}{|l|l|l|l|}
  \hline
  Cases & $M^1$ & $M^2$ & Parameters \\  \hline
  $14.1\quad \omega=0,\mu\neq 0$ & $\left(
                                       \begin{array}{cc}
                                          \frac{\mu^2}{\lambda} & \mu \\
          \mu & \lambda \\
                                       \end{array}
                                     \right)$
    & $\left(
        \begin{array}{cc}
          0& 0 \\
          0 & 0 \\
        \end{array}
      \right) $
     & $\lambda,\mu \in \k^{\times}$  \\
  \hline
   $14.2\quad \omega=0,\mu=0$ & $\left(
                                       \begin{array}{cc}
                                          0 & 0 \\
          0 & \lambda \\
                                       \end{array}
                                     \right)$
    & $\left(
        \begin{array}{cc}
          0& 0 \\
          0 & 0 \\
        \end{array}
      \right) $
     & $\lambda \in \k^{\times}$  \\
  \hline
  $14.3\quad \mu=\omega\neq 0$ & $\left(
        \begin{array}{cc}
          0 & \mu \\
          \mu & \lambda \\
        \end{array}
      \right) $& $\left(
                   \begin{array}{cc}
                     0 & 0 \\
                     0 & \mu \\
                   \end{array}
                 \right)$
       & $\lambda,\mu \in \k^{\times}$ \\
   \hline
   $14.4\quad \omega\neq 0,\mu=0$ & $\left(
        \begin{array}{cc}
          0 &  0\\
          0 & \lambda \\
        \end{array}
      \right) $& $\left(
                   \begin{array}{cc}
                     0 & 0 \\
                     0 & \omega \\
                   \end{array}
                 \right)$
       & $\lambda,\omega\in \k^{\times}$ \\
   \hline
   $14.5\quad \omega \neq 0, \mu\neq \omega, \mu\neq 0$ & $\left(
        \begin{array}{cc}
          \frac{\mu(\mu-\omega)}{\lambda} & \mu \\
          \mu & \lambda \\
        \end{array}
      \right) $& $\left(
                   \begin{array}{cc}
                     0 & 0 \\
                     0 & \omega \\
                   \end{array}
                 \right)$
       & $\lambda,\mu,\omega\in \k^{\times}$\\
   \hline
\end{tabular}.
\end{center}
For the cases listed above, we can choose corresponding $N^1,N^2$, such that there exists $A\in \mathrm{GL}_{\k}(2)$ such that $(\ref{eqs})$ holds. We have the following tabular:
\begin{center}
\begin{tabular}{|l|l|l|l|}
  \hline
  Cases & $N^1$ &$N^2$ & $A$ \\  \hline
  Case $14.1$ & $\left(
                                       \begin{array}{cc}
                                          1& 0\\
          0 & 0 \\
                                       \end{array}
                                     \right)$& $\left(
        \begin{array}{cc}
          0& 0 \\
          0 & 0 \\
        \end{array}
      \right)$
    & $\left(
        \begin{array}{cc}
          \frac{\lambda}{\mu^2}& -\lambda \\
          0 & \mu \\
        \end{array}
      \right) $\\
  \hline
   Case $14.2$ & $\left(
                                       \begin{array}{cc}
                                          0 & 0 \\
          0 & 1 \\
                                       \end{array}
                                     \right)$ & $\left(
        \begin{array}{cc}
          0 & 0 \\
          0 & 0 \\
        \end{array}
      \right)$
    & $\left(
        \begin{array}{cc}
          \lambda & 0 \\
          0 & 1 \\
        \end{array}
      \right) $\\
  \hline
  Case $14.3$ & $\left(
        \begin{array}{cc}
          0 & 1 \\
          1 & 1 \\
        \end{array}
      \right)$ &$\left(
                   \begin{array}{cc}
                     0 & 0 \\
                     0 & 1 \\
                   \end{array}
                 \right) $ & $\left(
                   \begin{array}{cc}
                     \frac{\lambda}{\mu^2} & 0 \\
                     0 & \frac{1}{\mu} \\
                   \end{array}
                 \right)$\\
   \hline
   Case $14.4$ & $\left(
        \begin{array}{cc}
          0 &  0\\
          0 & 1 \\
        \end{array}
      \right)$ & $\left(
                   \begin{array}{cc}
                     0 & 0 \\
                     0 & 1 \\
                   \end{array}
                 \right)$& $\left(
                   \begin{array}{cc}
                     \frac{\lambda}{\omega^2}& 0 \\
                     0 & \frac{1}{\omega} \\
                   \end{array}
                 \right)$\\
   \hline
  Case $14.5$ & $\left(
        \begin{array}{cc}
         1 & 1 \\
          1 & 0 \\
        \end{array}
      \right)$ & $\left(
                   \begin{array}{cc}
                     0 & 0 \\
                     0 & 1 \\
                   \end{array}
                 \right) $& $\left(
                   \begin{array}{cc}
                     \frac{\lambda}{\mu(\mu-\omega)} & -\frac{\lambda}{\mu\omega} \\
                     0 & \frac{1}{\omega}\\
                   \end{array}
                 \right)$\\
   \hline
\end{tabular}.
\end{center}
By Theorem \ref{isom}, $\mathcal{A}$ is isomorphic to $\mathcal{B}_5,\mathcal{B}_6,\mathcal{B}_7,\mathcal{B}_8$ and $\mathcal{B}_9$, when $M^1, M^2$ belong to Case $14.1$, Case $14.2$, Case $14.3$, Case $14.4$ and Case $14.5$, respectively.

For Case $15$, we divide it into the following $5$ cases:
\begin{center}
\begin{tabular}{|l|l|l|l|}
  \hline
  Cases & $M^1$ & $M^2$ & Parameters \\  \hline
  $15.1\quad \omega\neq 0,\mu\neq 0, \omega\lambda +\frac{\mu^2}{4}\neq 0 $ & $\left(
                                       \begin{array}{cc}
                                         \omega &  0 \\
                                         0 & \lambda\\
                                       \end{array}
                                     \right)$
    & $\left(
        \begin{array}{cc}
          0 & \omega \\
          \omega & \mu \\
        \end{array}
      \right)$
     & $\lambda,\omega,\mu \in \k^{\times}$  \\
  \hline
  $15.2\quad \omega\neq 0,\mu\neq 0, \omega\lambda +\frac{\mu^2}{4}=0 $ & $\left(
                                       \begin{array}{cc}
                                         \omega &  0 \\
                                         0 & \lambda\\
                                       \end{array}
                                     \right)$
    & $\left(
        \begin{array}{cc}
          0 & \omega \\
          \omega & \mu \\
        \end{array}
      \right)$
     & $\lambda,\omega,\mu \in \k^{\times}$  \\
     \hline
   $15.3\quad \omega=0,\mu=0$ & $\left(
                                       \begin{array}{cc}
                                          0 & 0 \\
          0 & \lambda \\
                                       \end{array}
                                     \right)$
    & $\left(
        \begin{array}{cc}
          0& 0 \\
          0 & 0 \\
        \end{array}
      \right) $
     & $\lambda \in \k^{\times}$  \\
  \hline
  $15.4\quad \mu=0,\omega\neq 0$ & $\left(
        \begin{array}{cc}
          \omega & 0 \\
          0 & \lambda \\
        \end{array}
      \right) $& $\left(
                   \begin{array}{cc}
                     0 & \omega \\
                     \omega &  0\\
                   \end{array}
                 \right)$
       & $\lambda,\omega \in \k^{\times}$ \\
   \hline
   $15.5\quad \mu\neq 0,\omega = 0$ & $\left(
        \begin{array}{cc}
          0 &  0\\
          0 & \lambda \\
        \end{array}
      \right) $& $\left(
                   \begin{array}{cc}
                     0 & 0 \\
                     0 & \mu \\
                   \end{array}
                 \right)$
       & $\lambda,\mu\in \k^{\times}$ \\
   \hline
\end{tabular}.
\end{center}
For the cases listed above, we can choose corresponding $N^1,N^2$, such that there exists $A\in \mathrm{GL}_{\k}(2)$ such that $(\ref{eqs})$ holds. We have the following tabular:
\begin{center}
\begin{tabular}{|l|l|l|l|}
  \hline
  Cases & $N^1$ & $N^2$ & $A$ \\  \hline
Case $15.1$ & $\left(
                                       \begin{array}{cc}
                                          1& 0\\
          0 & 1 \\
                                       \end{array}
                                     \right)$&$\left(
        \begin{array}{cc}
          0& 1 \\
          1 & 0 \\
        \end{array}
      \right)$
    & $\left(
        \begin{array}{cc}
          \frac{1}{\omega}& \frac{-\mu}{\omega \sqrt{4\omega\lambda+\mu^2}}  \\
          0 &  \frac{2}{\sqrt{4\omega\lambda+\mu^2}}\\
        \end{array}
      \right) $\\
  \hline
  Case $15.2$ & $\left(
                                       \begin{array}{cc}
                                          1 & 0 \\
          0 & \frac{-1}{4} \\
                                       \end{array}
                                     \right)$ &$\left(
        \begin{array}{cc}
          0 & 1\\
          1 & 1 \\
        \end{array}
      \right)$
    & $\left(
        \begin{array}{cc}
          \frac{1}{\omega} & 0\\
          0 & \frac{1}{\mu} \\
        \end{array}
      \right)
     $\\
  \hline
  Case $15.3$ & $\left(
        \begin{array}{cc}
          0 & 0\\
          0& 1 \\
        \end{array}
      \right)$ & $\left(
                   \begin{array}{cc}
                     0 & 0 \\
                     0 & 0 \\
                   \end{array}
                 \right) $ & $\left(
                   \begin{array}{cc}
                     \lambda & 0 \\
                     0 & 1 \\
                   \end{array}
                 \right)$\\
   \hline   Case $15.4$
                  & $\left(
        \begin{array}{cc}
         1 & 0 \\
          0 & 1 \\
        \end{array}
      \right)$ & $\left(
                   \begin{array}{cc}
                     0 & 1 \\
                     1 & 0 \\
                   \end{array}
                 \right) $& $\left(
                   \begin{array}{cc}
                     \frac{1}{\omega} & 0 \\
                     0 & \frac{1}{\sqrt{\lambda\omega}}\\
                   \end{array}
                 \right)$\\
   \hline
   Case $15.5$ & $\left(
        \begin{array}{cc}
          0 &  0\\
          0 & 1 \\
        \end{array}
      \right)$ &$\left(
                   \begin{array}{cc}
                     0 & 0 \\
                     0 & 1 \\
                   \end{array}
                 \right)$& $\left(
                   \begin{array}{cc}
                     \frac{\lambda}{\mu^2}& 0 \\
                     0 & \frac{1}{\mu} \\
                   \end{array}
                 \right)$\\
   \hline
\end{tabular}.
\end{center}
By Theorem \ref{isom}, $\mathcal{A}$ is isomorphic to $\mathcal{B}_{11},\mathcal{B}_{10},\mathcal{B}_6,\mathcal{B}_{11}$ and $\mathcal{B}_8$, when $M^1, M^2$ belong to Case $15.1$, Case $15.2$, Case $15.3$, Case $15.4$ and Case $15.5$, respectively.

For Case $13$, we have $$M^1=\left(
                                       \begin{array}{cc}
                                         \nu+\lambda(\mu-\omega) &  \mu \\
                                         \mu & \frac{\mu}{\lambda}\\
                                       \end{array}
                                     \right), M^2= \left(
        \begin{array}{cc}
          \lambda\nu & \nu\\
          \nu & \omega \\
        \end{array}
      \right)$$ with $\lambda,\mu,\nu \in \k^{\times}, \omega\in \k$. Let $$N^1=\left(
       \begin{array}{cc}
         1+\frac{\lambda(\mu-w)}{\nu} & 1\\
         1 & \frac{\nu}{\lambda \mu}\\
       \end{array}
     \right), N^2=\left(
       \begin{array}{cc}
         \frac{\mu\lambda}{\nu} & 1 \\
         1 & \frac{w}{\mu}\\
       \end{array}
     \right).$$ There exists $A=\left(
       \begin{array}{cc}
         \frac{1}{\nu} & 0 \\
         0 & \frac{1}{\mu}\\
       \end{array}
     \right)$ such that  $(\ref{eqs})$ holds. Let $s=\frac{\lambda\mu}{\nu}, t=\frac{\omega}{\mu}$. Then $$N^1=\left(
       \begin{array}{cc}
         1+s-st & 1\\
         1 & \frac{1}{s}\\
       \end{array}
     \right), N^2=\left(
       \begin{array}{cc}
         s & 1 \\
         1 & t\\
       \end{array}
     \right)$$ with  $s\in \k^{\times}, t\in \k$.
By Theorem \ref{isom}, $\mathcal{A}$ is isomorphic to $\mathcal{B}(s,t)$.
\end{proof}

\begin{prop}\label{stone}
For any $s\in \k^{\times}$, we have \begin{enumerate}
\item $\mathcal{B}(s,s^{-1})\cong \mathcal{B}_8$, if $s\neq -1$;
\item $\mathcal{B}(s,s^{-1})\cong \mathcal{B}_6$, if $s=-1$.
\end{enumerate}
\end{prop}
\begin{proof}
By definition, $\mathcal{B}_8$ and $\mathcal{B}(s,s^{-1})$ are defined by crisscrossed ordered $2$-tuples $(M^1,M^2)$ and $(N^1,N^2)$, where
\begin{align*}
M^1= \left(
       \begin{array}{cc}
         0 & 0\\
         0 & 1\\
       \end{array}
     \right)=M^2 \quad
\text{and}\quad N^1= \left(
       \begin{array}{cc}
         s & 1\\
         1 & s^{-1}\\
       \end{array}
     \right)=N^2.
\end{align*}
When $s\neq -1$, we can choose $A=(a_{ij})_{2\times 2}=\left(
       \begin{array}{cc}
         s & \frac{1}{s}+2\\
         s+1 & \frac{s+1}{s}\\
       \end{array}
     \right)$ in $\mathrm{GL}_{\k}(2)$ such that $(\ref{eqs})$ hold, i.e.,
     \begin{align*}
     \left(
       \begin{array}{cc}
         a_{11}s+a_{12}s & a_{11}+a_{12}\\
         a_{11}+a_{12} & a_{11}s^{-1}+a_{12}s^{-1}\\
         a_{21}s+a_{22}s & a_{21}+a_{22}\\
         a_{21}+a_{22} & a_{21}s^{-1}+a_{22}s^{-1}\\
       \end{array}
     \right)= \left(
       \begin{array}{cc}
         a_{21}^2 & a_{21}a_{22}\\
         a_{21}a_{22} & a_{22}^{2}\\
         a_{21}^2 & a_{21}a_{22}\\
         a_{21}a_{22} & a_{22}^2\\
       \end{array}
     \right).
     \end{align*}
     By Theorem \ref{isom}, $\mathcal{B}(s,s^{-1})\cong \mathcal{B}_8$, when $s\neq -1$.

     When $s=-1$, $\mathcal{B}(s,s^{-1})=\mathcal{B}(-1,-1)$ is defined by the crisscross ordered $2$-tuple $(N^1,N^2)$ of $2\times 2$ matrixes, where
   \begin{align*}
    N^1= \left(
       \begin{array}{cc}
         -1 & 1\\
         1 & -1\\
       \end{array}
     \right)=N^2.
\end{align*}
By definition, $\mathcal{B}_6$ is defined by the crisscrossed ordered $2$-tuple $(M^1,M^2)$ with
\begin{align*}
M^1= \left(
       \begin{array}{cc}
         0 & 0\\
         0 & 1\\
       \end{array}
     \right), M^2=\left(
       \begin{array}{cc}
         0 & 0\\
         0 & 0\\
       \end{array}
     \right).
\end{align*}
We can choose $A=(a_{ij})_{2\times 2}=\left(
       \begin{array}{cc}
         -1 & 0\\
         1 & -1\\
       \end{array}
     \right)$ in $\mathrm{GL}_{\k}(2)$ such that $(\ref{eqs})$ hold, i.e.,
       \begin{align*}
     \left(
       \begin{array}{cc}
         -a_{11}-a_{12}s & a_{11}+a_{12}\\
         a_{11}+a_{12} & -a_{11}-a_{12}\\
         -a_{21}-a_{22}s & a_{21}+a_{22}\\
         a_{21}+a_{22} & -a_{21}-a_{22}\\
       \end{array}
     \right)= \left(
       \begin{array}{cc}
         a_{21}^2 & a_{21}a_{22}\\
         a_{21}a_{22} & a_{22}^{2}\\
         0 & 0\\
         0 & 0\\
       \end{array}
     \right).
     \end{align*}
     By Theorem \ref{isom}, $\mathcal{B}(-1,-1)\cong \mathcal{B}_6$.
\end{proof}

\section{cohomology graded algebras of DG free algebras}
In this section, we will compute the cohomology graded algebra of DG free algebras with two degree one generators.
Due to the classifications finished in the previous section, we only need to compute $$H(\mathcal{B}_1),H(\mathcal{B}_2),\cdots, H(\mathcal{B}_{11}), H(\mathcal{B}(s,t))\,\,\text{with}\,\,st\neq 1, s\in \k^{\times},$$ by
Proposition \ref{nonsym}, Proposition \ref{symone}, Proposition \ref{symtwo} and Proposition \ref{stone}. For any cochain DG algebra $\mathcal{A}$ and cocycle element $z\in \mathrm{ker}(\partial_{\mathcal{A}}^i)$, we write $\lceil z \rceil$ as the cohomology class in $H(\mathcal{A})$ represented by $z$. The following proposition gives all possible cohomology graded algebras for non-trivial DG free algebras with two degree one generators.

\begin{prop}\label{cohomology}
We have \begin{enumerate}
\item $H(\mathcal{B}_1)=\k$;
\item $H(\mathcal{B}_2)=\k$;
\item $H(\mathcal{B}_3)=\k$;
\item $H(\mathcal{B}_4)=\k$;
\item $H(\mathcal{B}_5)=\k[\lceil x_2\rceil]$;
\item $H(\mathcal{B}_6)=\k[\lceil x_2\rceil,\lceil x_1x_2+x_2x_1\rceil]/(\lceil x_2\rceil^2)$;
\item $H(\mathcal{B}_7)=\k $;
\item $H(\mathcal{B}_8)=\k[ \lceil x_1-x_2\rceil] $;
\item $H(\mathcal{B}_9)=\k$;
\item $H(\mathcal{B}_{10})=\k$;
\item $H(\mathcal{B}_{11})=\k$;
\item $H(\mathcal{B}(s,t))=\k$, when $st\neq 1$ and $s\in \k^{\times}$.
\end{enumerate}
\end{prop}

\begin{proof}
(1)It is easy to see that $H^0(\mathcal{B}_1)=\k$ and $H^1(\mathcal{B}_1)=0$ since $\partial_{\mathcal{B}_1}(x_1)=x_1^2$ and $\partial_{\mathcal{B}_1}(x_2)=x_2x_1$. One only need to prove that $H^i(\mathcal{B}_1)=0$, for any $i\ge 2$. For any cocycle element in $\mathcal{B}_1^i$, we may write it as $a_1x_1+a_2x_2$ for some $a_1,a_2\in \mathcal{B}_1^{i-1}$. We have
\begin{align*}
0=&\partial_{\mathcal{B}_1}[a_1x_1+a_2x_2]\\
=&\partial_{\mathcal{B}_1}(a_1)x_1+(-1)^ia_1x_1^2+\partial_{\mathcal{B}_1}(a_2)x_2+(-1)^ia_2x_2x_1\\
=&[\partial_{\mathcal{B}_1}(a_1)+(-1)^ia_1x_1+(-1)^ia_2x_2]x_1+\partial_{\mathcal{B}_1}(a_2)x_2.
\end{align*}
Hence $\partial_{\mathcal{B}_1}(a_1)+(-1)^ia_1x_1+(-1)^ia_2x_2=0$ and $\partial_{\mathcal{B}_1}(a_2)=0$. Then $$\partial_{\mathcal{B}_1}[(-1)^{i-1}a_1]=a_1x_1+a_2x_2.$$ So $H^{i}(\mathcal{B}_1)=0$, for any $i\ge 2$.  Hence $H(\mathcal{B}_1)=\k$.

(2)Since $\partial_{\mathcal{B}_2}(x_1)=x_1^2$ and $\partial_{\mathcal{B}_2}(x_2)=x_1x_2$, one sees that $H^0(\mathcal{B}_2)=\k$ and $H^1(\mathcal{B}_2)=0$. It suffices to show that $H^i(\mathcal{B}_2)=0$ for any $i\ge 2$. For any cocycle element in $\mathcal{B}_2^i$, we may write it as $x_1a_1+x_2a_2$ for some $a_1,a_2\in \mathcal{B}_2^{i-1}$. We have
\begin{align*}
0=&\partial_{\mathcal{B}_2}[x_1a_1+x_2a_2]\\
=&x_1^2a_1-x_1\partial_{\mathcal{B}_2}(a_1)+x_1x_2a_2-x_2\partial_{\mathcal{B}_2}(a_2)\\
=&x_1[x_1a_1-\partial_{\mathcal{B}_2}(a_1)+x_2a_2]-x_2\partial_{\mathcal{B}_2}(a_2).
\end{align*}
Hence $\partial_{\mathcal{B}_2}(a_1)=x_1a_1+x_2a_2$ and $\partial_{\mathcal{B}_2}(a_2)=0$. So $H^{i}(\mathcal{B}_2)=0$, for any $i\ge 2$.  Hence $H(\mathcal{B}_2)=\k$.

(3)The differential of $\mathcal{B}_3$ is defined by $\partial_{\mathcal{B}_3}(x_1)=x_1^2, \partial_{\mathcal{B}_3}(x_2)=x_1x_2+x_2x_1$. Obviously, we have $H^0(\mathcal{B}_3)=\k$ and $H^1(\mathcal{B}_3)=0$. For any cocycle element in $\mathcal{B}_3^i$, we may write it as $x_1a_1+x_2a_2$ for some $a_1,a_2\in \mathcal{B}_1^{i-1}$. Since
\begin{align*}
0=&\partial_{\mathcal{B}_3}[x_1a_1+x_2a_2]\\
=&x_1^2a_1-x_1\partial_{\mathcal{B}_3}(a_1)+(x_1x_2+x_2x_1)a_2-x_2\partial_{\mathcal{B}_3}(a_2)\\
=&x_1[x_1a_1-\partial_{\mathcal{B}_3}(a_1)+x_2a_2]+x_2(x_1a_2-\partial_{\mathcal{B}_3}(a_2)),
\end{align*}
we have $\partial_{\mathcal{B}_3}(a_1)=x_1a_1+x_2a_2$ and $\partial_{\mathcal{B}_3}(a_2)=x_1a_2$. So $H^{i}(\mathcal{B}_3)=0$, for any $i\ge 2$.  Hence $H(\mathcal{B}_3)=\k$.

(4)Since $\partial_{\mathcal{B}_4}(x_1)=x_1^2$ and $\partial_{\mathcal{B}_4}(x_2)=x_2^2$, one sees that $H^0(\mathcal{B}_2)=\k$ and $H^1(\mathcal{B}_2)=0$. We should prove that $H^i(\mathcal{B}_4)=0$, for any $i\ge 2$.  For any cocycle element in $\mathcal{B}_4^i$, we may write it as $x_1a_1+x_2a_2$ for some $a_1,a_2\in \mathcal{B}_4^{i-1}$. We have
\begin{align*}
0=&\partial_{\mathcal{B}_4}[x_1a_1+x_2a_2]\\
=&x_1^2a_1-x_1\partial_{\mathcal{B}_4}(a_1)+x_2^2a_2-x_2\partial_{\mathcal{B}_4}(a_2)\\
=&x_1[x_1a_1-\partial_{\mathcal{B}_4}(a_1)]+x_2[x_2a_2-\partial_{\mathcal{B}_4}(a_2)].
\end{align*}
Hence $\partial_{\mathcal{B}_4}(a_1)=x_1a_1$ and $\partial_{\mathcal{B}_4}(a_2)=x_2a_2$. Since $\partial_{\mathcal{B}_4}(a_1+a_2)=x_1a_1+x_2a_2$, we have $H^i(\mathcal{B}_4)=0$ for any $i\ge 2$. So $H(\mathcal{B}_4)=\k$.

(5)The differential of $\mathcal{B}_5$ is defined by $\partial_{\mathcal{B}_5}(x_1)=x_1^2, \partial_{\mathcal{B}_5}(x_2)=0$. Obviously, we have $H^0(\mathcal{B}_5)=\k$ and $H^1(\mathcal{B}_5)=\k\lceil x_2\rceil$. We want to show inductively that $H^j(\mathcal{B}_5)=\k\lceil x_2^j\rceil$ for any $j\ge 1$. Suppose that we have proved that $H^{i}(\mathcal{B}_5)=\k\lceil x_2^i\rceil$. For any cocycle element in $\mathcal{B}_5^{i+1}$, we may write it as $x_1a_1+x_2a_2$, $a_1,a_2\in \mathcal{B}_5^i$. Since
\begin{align*}
0=&\partial_{\mathcal{B}_5}[x_1a_1+x_2a_2]\\
=&x_1^2a_1-x_1\partial_{\mathcal{B}_5}(a_1)-x_2\partial_{\mathcal{B}_5}(a_2)\\
=&x_1[x_1a_1-\partial_{\mathcal{B}_5}(a_1)]-x_2\partial_{\mathcal{B}_5}(a_2),
\end{align*}
we have $x_1a_1=\partial_{\mathcal{B}_5}(a_1)$ and $\partial_{\mathcal{B}_5}(a_2)=0$. By the induction hypothesis, we have $H^i(\mathcal{B}_5)=\k\lceil x_2^i\rceil$. Hence there exists some $l\in \k$ such that $\lceil a_2\rceil =l\lceil x_2^i\rceil$.
  So $$\lceil x_1a_1+x_2a_2\rceil=\lceil \partial_{\mathcal{B}_5}(a_1) +x_2a_2\rceil = \lceil x_2a_2\rceil=\lceil x_2\rceil \cdot\lceil a_2\rceil =l\lceil x_2^{i+1}\rceil. $$
Then $H^{i+1}(\mathcal{B}_5)=\k\lceil x_2^{i+1}\rceil$. By the induction above, we get $H(\mathcal{B}_5)=\k[\lceil x_2\rceil]$.

(6)The differential of $\mathcal{B}_6$ is defined by $\partial_{\mathcal{B}_6}(x_1)=x_2^2, \partial_{\mathcal{B}_6}(x_2)=0$. It is easy for one to see that $H^0(\mathcal{B}_6)=\k$ and $H^1(\mathcal{B}_6)=\k\lceil x_2\rceil$. Since $\mathcal{B}_6^2=\k x_1^2\oplus \k x_1x_2\oplus \k x_2x_1\oplus \k x_2^2$, any cocycle element in $\mathcal{B}_6^2$ can be written as $l_1x_1^2+l_2x_1x_2+l_3x_2x_1+l_4x_2^2$ for some $l_1,l_2,l_3,l_4\in \k$. We have
\begin{align*}
0&=\partial_{\mathcal{B}_6}(l_1x_1^2+l_2x_1x_2+l_3x_2x_1+l_4x_2^2)\\
&=l_1x_2^2x_1-l_1x_1x_2^2+l_2x_2^3-l_3x_2^3\\
&=l_1(x_2^2x_1-x_1x_2^2)+(l_2-l_3)x_2^3.
\end{align*}
We have $l_1=0$ and $l_2=l_3$. Hence $\mathrm{ker}(\partial_{\mathcal{B}_6}^2)= \k(x_1x_2+x_2x_1)\oplus \k x_2^2$. By the definition of $\partial_{\mathcal{B}_6}$, one sees that $\mathrm{im}(\partial_{\mathcal{B}_6}^1)=\k x_2^2$. So $H^2(\mathcal{B}_6)=\k \lceil x_1x_2+x_2x_1\rceil$.
Assume that we have proved $H^{2i-1}(\mathcal{B}_6)=\k \lceil x_2(x_1x_2+x_2x_1)^{i-1}\rceil$ and $H^{2i}(\mathcal{B}_6)=\k \lceil (x_1x_2+x_2x_1)^i\rceil$, $i\ge 1$. We should show that \begin{align*}
H^{2i+1}(\mathcal{B}_6)&=\k \lceil x_2(x_1x_2+x_2x_1)^i\rceil  \\
 \text{and}\,\, H^{2i+2}(\mathcal{B}_6)&=\k \lceil (x_1x_2+x_2x_1)^{i+1}\rceil.
\end{align*}
 For any cocycle element in $\mathcal{B}_6^{2i+1}$, we may write it as $x_1a_1+x_2a_2$, $a_1,a_2\in \mathcal{B}_6^{2i}$.
 We have \begin{align*}
 0=&\partial_{\mathcal{B}_6}(x_1a_1+x_2a_2)\\
 =&x_2^2a_1-x_1\partial_{\mathcal{B}_6}(a_1)-x_2\partial_{\mathcal{B}_6}(a_2) \\
 =&x_2[x_2a_1-\partial_{\mathcal{B}_6}(a_2)]-x_1\partial_{\mathcal{B}_6}(a_1).
 \end{align*}
 Then $\partial_{\mathcal{B}_6}(a_2)=x_2a_1$ and $\partial_{\mathcal{B}_6}(a_1)=0$. Since $Z^{2i}(\mathcal{B}_6)\cong H^{2i}(\mathcal{B}_6)\oplus B^{2i}(\mathcal{B}_6)$, we may let $a_1=l(x_1x_2+x_2x_1)^i+\partial_{\mathcal{B}_6}(\chi)$ for some $l\in \k$ and $\chi\in \mathcal{B}_6^{2i-1}$.
 We have $\partial_{\mathcal{B}_6}(a_2)=lx_2(x_1x_2+x_2x_1)^i+x_2\partial_{\mathcal{B}_6}(\chi)$. This implies that $l=0$ and $a_2=-x_2\chi+t(x_1x_2+x_2x_1)^i+\partial_{\mathcal{B}_6}(\lambda)$, for some $t\in \k$ and $\lambda\in \mathcal{B}_6^{2i-1}$. Then
 \begin{align*}
 \lceil x_1a_1+x_2a_2\rceil & =\lceil x_1 \partial_{\mathcal{B}_6}(\chi)-x_2^2\chi +tx_2(x_1x_2+x_2x_1)^i+x_2\partial_{\mathcal{B}_6}(\lambda)\rceil \\
 &=\lceil  \partial_{\mathcal{B}_6}(-x_1\chi-x_2\lambda)  +tx_2(x_1x_2+x_2x_1)^i              \rceil \\
 &=t\lceil x_2(x_1x_2+x_2x_1)^i   \rceil.
 \end{align*}
 So $H^{2i+1}(\mathcal{B}_6)=\k\lceil x_2(x_1x_2+x_2x_1)^i   \rceil$.

 Now, let $x_1b_1+x_2b_2$ be an arbitrary cocycle element in $\mathcal{B}_6^{2i+2}$. Then we have
   \begin{align*}
 0=&\partial_{\mathcal{B}_6}(x_1b_1+x_2b_2)\\
 =&x_2^2b_1-x_1\partial_{\mathcal{B}_6}(b_1)-x_2\partial_{\mathcal{B}_6}(b_2) \\
 =&x_2[x_2b_1-\partial_{\mathcal{B}_6}(b_2)]-x_1\partial_{\mathcal{B}_6}(b_1).
 \end{align*}
 So $\partial_{\mathcal{B}_6}(b_2)=x_2b_1$ and $\partial_{\mathcal{B}_6}(b_1)=0$. Since $Z^{2i+1}(\mathcal{B}_6)\cong H^{2i+1}(\mathcal{B}_6)\oplus B^{2i+1}(\mathcal{B}_6)$, we may let $b_1=rx_2(x_1x_2+x_2x_1)^i+\partial_{\mathcal{B}_6}(\omega)$ for some $r\in \k$ and $\omega\in \mathcal{B}_6^{2i}$.
 Then $\partial_{\mathcal{B}_6}(b_2)=rx_2^2(x_1x_2+x_2x_1)^i+x_2\partial_{\mathcal{B}_6}(\omega)$. So $$b_2=rx_1(x_1x_2+x_2x_1)^i-x_2\omega+sx_2(x_1x_2+x_2x_1)^i+\partial_{\mathcal{B}_6}(\varphi)$$ for some $s\in \k$ and $\varphi\in \mathcal{B}_6^{2i}$
 Therefore,
 \begin{align*}
 \lceil x_1b_1+x_2b_2\rceil&= \lceil r(x_1x_2+x_2x_1)^{i+1}+\partial_{\mathcal{B}_6}[-x_1\omega+sx_1(x_1x_2+x_2x_1)^i-x_2\varphi]\rceil\\
                           &= r\lceil (x_1x_2+x_2x_1)^{i+1}\rceil.
 \end{align*}
 Thus $H^{2i+2}(\mathcal{B}_6)=\k\lceil (x_1x_2+x_2x_1)^{i+1}   \rceil$. By the induction above, we obtain that
 \begin{align*}
H^{2n-1}(\mathcal{B}_6)&=\k \lceil x_2(x_1x_2+x_2x_1)^{n-1}\rceil  \\
 \text{and}\,\, H^{2n}(\mathcal{B}_6)&=\k \lceil (x_1x_2+x_2x_1)^{n}\rceil, \forall n\ge 1.
\end{align*}
 Since $x_2(x_1x_2+x_2x_1)-(x_1x_2+x_2x_1)x_2=x_2^2x_1-x_1x_2^2=\partial_{\mathcal{B}_6}(x_1^2)$, we have
 $$\lceil x_2\rceil \cdot \lceil x_1x_2+x_2x_1\rceil = \lceil x_1x_2+x_2x_1\rceil\cdot \lceil x_2\rceil$$ in $H(\mathcal{B}_6)$. Hence $H(\mathcal{B}_6)=\k[\lceil x_2\rceil,\lceil x_1x_2+x_2x_1\rceil]/(\lceil x_2\rceil^2)$.

 (7)Since $\partial_{\mathcal{B}_7}(x_1)=x_1x_2+x_2x_1+x_2^2$ and $\partial_{\mathcal{B}_7}(x_2)=x_2^2$, one sees that $H^0(\mathcal{B}_7)=\k$ and $H^1(\mathcal{B}_7)=0$. It suffices to show that $H^i(\mathcal{B}_7)=0$ for any $i\ge 2$. For any cocycle element in $\mathcal{B}_7^i$, we may write it as $x_1a_1+x_2a_2$ for some $a_1,a_2\in \mathcal{B}_7^{i-1}$. We have
\begin{align*}
0=&\partial_{\mathcal{B}_7}[x_1a_1+x_2a_2]\\
=&(x_1x_2+x_2x_1+x_2^2)a_1-x_1\partial_{\mathcal{B}_7}(a_1)+x_2^2a_2-x_2\partial_{\mathcal{B}_7}(a_2)\\
=&x_1[x_2a_1-\partial_{\mathcal{B}_7}(a_1)]+x_2[x_1a_1+x_2a_1+x_2a_2-\partial_{\mathcal{B}_7}(a_2)].
\end{align*}
Hence $\partial_{\mathcal{B}_7}(a_1)=x_2a_1$ and $\partial_{\mathcal{B}_2}(a_2)=x_1a_1+x_2a_1+x_2a_2$. Then $$\partial_{\mathcal{B}_2}(a_2-a_1)=x_1a_1+x_2a_2.$$  So $H^{i}(\mathcal{B}_2)=0$, for any $i\ge 2$.  Hence $H(\mathcal{B}_2)=\k$.

(8)The differential of $\mathcal{B}_8$ is defined by $\partial_{\mathcal{B}_8}(x_1)=x_2^2, \partial_{\mathcal{B}_8}(x_2)=x_2^2$. It is easy for one to see that $H^0(\mathcal{B}_8)=\k$ and $H^1(\mathcal{B}_8)=\k\lceil x_2-x_1\rceil$. We want to show inductively that $H^j(\mathcal{B}_8)=\k\lceil (x_2-x_1)^j\rceil$ for any $j\ge 1$. Suppose that we have proved that $H^{i}(\mathcal{B}_8)=\k\lceil (x_2-x_1)^i\rceil$. For any cocycle element in $\mathcal{B}_8^{i+1}$, we may write it as $x_1a_1+x_2a_2$, $a_1,a_2\in \mathcal{B}_8^i$. Since
\begin{align*}
0=&\partial_{\mathcal{B}_8}[x_1a_1+x_2a_2]\\
=&x_2^2a_1-x_1\partial_{\mathcal{B}_8}(a_1)+x_2^2a_2-x_2\partial_{\mathcal{B}_8}(a_2)\\
=&-x_1\partial_{\mathcal{B}_8}(a_1)+x_2[x_2a_1+x_2a_2-\partial_{\mathcal{B}_8}(a_2)],
\end{align*}
we have $x_2a_1+x_2a_2=\partial_{\mathcal{B}_8}(a_2)$ and $\partial_{\mathcal{B}_8}(a_1)=0$. By the induction hypothesis, we have $H^i(\mathcal{B}_8)=\k\lceil (x_2-x_1)^i\rceil$. So $\lceil a_1\rceil =l\lceil (x_2-x_1)^i\rceil$, for some $l\in \k$.
  So \begin{align*}
  \lceil x_1a_1+x_2a_2\rceil & =\lceil x_2a_1+x_2a_2 +x_1a_1-x_2a_1 \rceil\\
   &= \lceil \partial_{\mathcal{B}_8}(a_2)+(x_1-x_2)a_1\rceil\\
   &=\lceil (x_1-x_2)a_1 \rceil \\
   &=l\lceil (x_2-x_1)^{i+1}\rceil
   \end{align*}
and $H^{i+1}(\mathcal{B}_8)=\k\lceil (x_2-x_1)^{i+1}\rceil$. By the induction above, we have $$H(\mathcal{B}_8)=\k[\lceil x_2-x_1\rceil].$$

(9)We have $\partial_{\mathcal{B}_9}(x_1)=x_1^2+x_1x_2+x_2x_1$ and $\partial_{\mathcal{B}_9}(x_2)=x_2^2$. Obviously, $H^0(\mathcal{B}_9)=\k$ and $H^1(\mathcal{B}_9)=0$. Any cocycle element in $\mathcal{B}_9^i$ can be written by $x_1a_1+x_2a_2$ for some $a_1,a_2\in \mathcal{B}_9^{i-1}, i\ge 2$. We have
\begin{align*}
0=&\partial_{\mathcal{B}_9}[x_1a_1+x_2a_2]\\
=&(x_1^2+x_1x_2+x_2x_1)a_1-x_1\partial_{\mathcal{B}_9}(a_1)+x_2^2a_2-x_2\partial_{\mathcal{B}_9}(a_2)\\
=&x_1[x_1a_1+x_2a_1-\partial_{\mathcal{B}_9}(a_1)]+x_2[x_1a_1+x_2a_2-\partial_{\mathcal{B}_9}(a_2)].
\end{align*}
Then $\partial_{\mathcal{B}_9}(a_1)=x_1a_1+x_2a_1$ and $\partial_{\mathcal{B}_9}(a_2)=x_1a_1+x_2a_2$, which implies that $H^i(\mathcal{B}_9)=0$ and hence $H(\mathcal{B}_9)=\k$.

(10)The differential of $\mathcal{B}_{10}$ is defined by $$\partial_{\mathcal{B}_{10}}(x_1)=x_1^2-\frac{1}{4}x_2^2, \quad \partial_{\mathcal{B}_{10}}(x_2)=x_1x_2+x_2x_1+x_2^2.$$ Obviously,  $H^0(\mathcal{B}_{10})=\k$ and $H^1(\mathcal{B}_{10})=0$. For any $x_1a_1+x_2a_2\in Z^{i}(\mathcal{B}_{10}), i\ge 2$, we have
\begin{align*}
0=&\partial_{\mathcal{B}_{10}}[x_1a_1+x_2a_2]\\
=&(x_1^2-\frac{1}{4}x_2^2)a_1-x_1\partial_{\mathcal{B}_{10}}(a_1)+[x_1x_2+x_2x_1+x_2^2]a_2-x_2\partial_{\mathcal{B}_{10}}(a_2)\\
=&x_1[x_1a_1+x_2a_2-\partial_{\mathcal{B}_{10}}(a_1)]+x_2[-\frac{1}{4}x_2a_1+x_2a_2-\partial_{\mathcal{B}_{10}}(a_2)].
\end{align*}
Then $\partial_{\mathcal{B}_{10}}(a_1)=x_1a_1+x_2a_2$ and $\partial_{\mathcal{B}_{10}}(a_2)=-\frac{1}{4}x_2a_1+x_2a_2$. So $H^i(\mathcal{B}_{10})=0$.

(11)We have $\partial_{\mathcal{B}_{11}}(x_1)=x_1^2+x_2^2$ and $\partial_{\mathcal{B}_{11}}(x_2)=x_1x_2+x_2x_1$. It is easy to check that $H^0(\mathcal{B}_{11})=\k$ and $H^1(\mathcal{B}_{11})=0$.  For any $x_1a_1+x_2a_2\in Z^{i}(\mathcal{B}_{11}), i\ge 2$, we have
\begin{align*}
0=&\partial_{\mathcal{B}_{11}}[x_1a_1+x_2a_2]\\
=&(x_1^2+x_2^2)a_1-x_1\partial_{\mathcal{B}_{11}}(a_1)+[x_1x_2+x_2x_1]a_2-x_2\partial_{\mathcal{B}_{11}}(a_2)\\
=&x_1[x_1a_1+x_2a_2-\partial_{\mathcal{B}_{11}}(a_1)]+x_2[x_2a_1+x_1a_2-\partial_{\mathcal{B}_{11}}(a_2)].
\end{align*}
Thus $\partial_{\mathcal{B}_{11}}(a_1)=x_1a_1+x_2a_2$ and $\partial_{\mathcal{B}_{11}}(a_2)=x_2a_1+x_1a_2$. Then $H^i(\mathcal{B}_{11})=0$ and hence $H(\mathcal{B}_{11})=\k$.

(12)The differential of $\mathcal{B}_{st}$ is defined by
\begin{align*}
\begin{cases}
\partial_{\mathcal{B}_{st}}(x_1)=(1+s-st)x_1^2+x_1x_2+x_2x_1+\frac{1}{s}x_2^2\\
\partial_{\mathcal{B}_{st}}(x_2)=sx_1^2+x_1x_2+x_2x_1+tx_2^2,
\end{cases}
\end{align*}
where $s\in \k^{\times}, t\in \k$ and $st\neq 1$. Obviously, $H^0(\mathcal{B}_{st})=\k$. Let $l_1x_1+l_2x_2$ be a cocycle element in $\mathcal{B}_{st}^1$. Then \begin{align*}
0=&\partial_{\mathcal{B}_{st}}(l_1x_1+l_2x_2)\\
=&l_1[(1+s-st)x_1^2+x_1x_2+x_2x_1+\frac{1}{s}x_2^2]+l_2[sx_1^2+x_1x_2+x_2x_1+tx_2^2]\\
=&[l_1(1+s-st)+l_2s]x_1^2+(l_1+l_2)(x_1x_2+x_2x_1)+(\frac{l_1}{s}+l_2t)x_2^2.
\end{align*}
We have \begin{align*}
\begin{cases}
l_1(1+s-st)+l_2s=0\\
l_1+l_2=0\\
\frac{l_1}{s}+l_2t=0.
\end{cases}
\end{align*}
Then $l_1=l_2=0$ and hence $H^1(\mathcal{B}_{st})=0$. For any $i\ge 2$, any element in $Z^i(\mathcal{B}_{st})$ can be written by
$x_1a_1+x_2a_2$, for some $a_1,a_2\in \mathcal{B}_{st}^{i-1}$. And we have
\begin{align*}
&0=\partial_{\mathcal{B}_{st}}(x_1a_1+x_2a_2)\\
&=[(1+s-st)x_1^2+x_1x_2+x_2x_1+\frac{1}{s}x_2^2]a_1-x_1\partial_{\mathcal{B}_{st}}(a_1)\\
&\quad\quad +[sx_1^2+x_1x_2+x_2x_1+tx_2^2]a_2-x_2\partial_{\mathcal{B}_{st}}(a_2)\\
&=x_1[(1+s-st)x_1a_1+x_2a_1-\partial_{\mathcal{B}_{st}}(a_1)+sx_1a_2+x_2a_2]\\
&\quad\quad +x_2[x_1a_1+\frac{1}{s}x_2a_1+x_1a_2+tx_2a_2-\partial_{\mathcal{B}_{st}}(a_2)].
\end{align*}
Hence \begin{align*}
\begin{cases}
(1+s-st)x_1a_1+x_2a_1-\partial_{\mathcal{B}_{st}}(a_1)+sx_1a_2+x_2a_2=0 \quad (1)\\
x_1a_1+\frac{1}{s}x_2a_1+x_1a_2+tx_2a_2-\partial_{\mathcal{B}_{st}}(a_2)=0  \quad\quad\quad\quad\quad\,\,\, (2)
\end{cases}.
\end{align*}
$(1)-(2)\times s$ implies that $(1-st)x_1a_1+(1-st)x_2a_2-\partial_{\mathcal{B}_{st}}(a_1-sa_2)=0.$ Then
$$\partial_{\mathcal{B}_{st}}[\frac{a_1}{1-st}-\frac{sa_2}{1-st}]=x_1a_1+x_2a_2.$$ Thus $H^{i}(\mathcal{B}_{st})=0$ and $H(\mathcal{B}_{st})=\k$.
\end{proof}


\section{Homological properties of DG free algebras}
For any $k$-vector space $V$, we write $V^*=\Hom_{k}(V,k)$. Let $\{e_i|i\in I\}$ be a basis of a finite dimensional $k$-vector space $V$.  We denote the dual basis of $V$ by $\{e_i^*|i\in I\}$, i.e., $\{e_i^*|i\in I\}$ is a basis of $V^*$ such that $e_i^*(e_j)=\delta_{i,j}$. For any graded vector space $W$ and $j\in\Bbb{Z}$,  the $j$-th suspension $\Sigma^j W$ of $W$ is a graded vector space defined by $(\Sigma^j W)^i=W^{i+j}$.

In this section, we will study homological properties of DG free algebras.
Now, let us review some fundamental homological properties for DG algebras.
\begin{defn}\label{basicdef}
{\rm Let $\mathcal{A}$ be a connected cochain DG algebra.
\begin{enumerate}
\item  If $\dim_{k}H(R\Hom_{\mathcal{A}}(k,\mathcal{A}))=1$, then $A$ is called Gorenstein (cf. \cite{FHT1,FM,Gam});
\item  If ${}_{\mathcal{A}}k$, or equivalently ${}_{\mathcal{A}^e}\mathcal{A}$, has a minimal semi-free resolution with a semi-basis concentrated in degree $0$, then $\mathcal{A}$ is called Koszul (cf. \cite{HW});
\item If ${}_{\mathcal{A}}k$, or equivalently the DG $\mathcal{A}^e$-module $\mathcal{A}$ is compact, then $\mathcal{A}$ is called homologically smooth (cf. \cite[Corollary 2.7]{MW3});
\item If $\mathcal{A}$ is homologically smooth and $$R\Hom_{\mathcal{A}^e}(\mathcal{A}, \mathcal{A}^e)\cong
\Sigma^{-n}\mathcal{A}$$ in  the derived category $\mathrm{D}((\mathcal{A}^e)^{op})$ of right DG $\mathcal{A}^e$-modules, then $\mathcal{A}$ is called an $n$-Calabi-Yau DG algebra  (cf. \cite{Gin,VdB}).
\end{enumerate}}
 \end{defn}
A natural question is whether DG free algebras have the properties listed in Definition \ref{basicdef}.
It is reasonable for us to consider the cases from easy to difficult. We have the following proposition for trivial DG free algebras.
\begin{prop}\label{zerodiff}
Let $\mathcal{A}$ be a connected cochain DG algebra such that $$H(\mathcal{A})=\k\langle \lceil y_1\rceil, \cdots, \lceil y_n\rceil \rangle,$$ for some degree $1$ cocycle elements $y_1,\cdots, y_n$  in $\mathcal{A}$. Then $\mathcal{A}$ is not a Gorenstein DG algebra but a Koszul and homologically smooth DG algebra.
\end{prop}
\begin{proof}
The graded module ${}_{H(A)}\k$ has the following minimal graded free resolution:
 \begin{align*}
  0\to H(A)\otimes (\bigoplus\limits_{i=1}^n\k e_{y_i}) \stackrel{d_1}{\to} H(A)\stackrel{\varepsilon}{\to} \k\to 0,
\end{align*}
where $\mu$ and $\partial_1$ are defined by $\mu(\lceil a\rceil \otimes \lceil b\rceil)= \lceil ab\rceil $, $\forall \lceil a\rceil \in H(\mathcal{A}), \lceil b\rceil \in H(\mathcal{A})^{op}$; $d_1(e_{y_i})=y_i,i=1,2,\cdots, n$.
Applying the constructing procedure of Eilenberg-Moore resolution, we can construct a minimal semi-free resolution $F$ of the DG $\mathcal{A}$-module $\k$. We have
$$F^{\#}=\mathcal{A}^{\#}\oplus [\mathcal{A}^{\#}\otimes (\bigoplus_{i=1}^n \k \Sigma e_{y_i})]$$ and
$\partial_{F}$ is defined by
$\partial_{F}(\Sigma e_{y_i})=y_{i}, i=1,2,\cdots, n.$ Hence $\mathcal{A}$ is a Koszul and homologically smooth DG algebra.
The DG $\mathcal{A}^{op}$-module $\Hom_{\mathcal{A}}(F,\mathcal{A})$ is a minimal semi-free DG module whose underlying graded module is
\begin{align*}
\{\k1^*\oplus [\bigoplus_{i=1}^n \k(\Sigma e_{y_i})^*]\}\otimes \mathcal{A}^{\#}.
\end{align*}
Hence  $\Hom_{\mathcal{A}}(F,\mathcal{A})$ is concentrated in degrees $\ge 0$. On the other hand, we have $\partial_{\Hom}[(\Sigma e_{y_i})^*]=0, \forall i\in \{1,2,\cdots, n\}$ and $\partial_{\Hom}(1^*)=\sum\limits_{i=1}^n-(\Sigma e_{y_i})^* y_i$, since the differential $\partial_{\Hom}$ of $\Hom_{\mathcal{A}}(F,\mathcal{A})$ is defined by $$\partial_{\Hom}(f)=\partial_{\mathcal{A}} \circ f -(-1)^{j}f\circ \partial_F, \forall  f\in [\Hom_{\mathcal{A}}(F,\mathcal{A})]^j.$$
 So $H^0(\Hom_{\mathcal{A}}(F,\mathcal{A}))=\bigoplus\limits_{i=1}^n\k(\Sigma e_{y_i})^*$ and hence $\mathcal{A}$ is not Gorenstein.
\end{proof}
\begin{rem}
By Theorem \ref{zerodiff}, any trivial DG free algebra is not Gorenstein but a Koszul and homologically smooth DG algebra. The following proposition indicates that is not Calabi-Yau.
\end{rem}

\begin{prop}\label{cygor}
Any Calabi-Yau connected cochain DG algebra $\mathcal{A}$ is a
Gorenstein DG algebra.
\end{prop}

\begin{proof}
Since $\mathcal{A}$ is a Calabi-Yau  connected cochain DG algebra, $\mathcal{A}$ is homologically smooth and
$$\Omega = R\Hom_{\mathcal{A}^e}(\mathcal{A},\mathcal{A}^e) \cong \Sigma^i\mathcal{A}$$ in $\mathrm{D}(\mathcal{A}^e)$, for some $i\in\Bbb{Z}$. Let $F$ be the minimal semi-free resolution of ${}_{\mathcal{A}}\k$. Then $F$ admits a finite semi-basis. Let $I$ be a semi-injective resolution of the DG $\mathcal{A}$-module $\mathcal{A}$. One sees that $\Hom_{\k}(F,I)$ is a homotopically injective DG $\mathcal{A}^e$-module. By \cite[Lemma 2.4]{MW3}, we have
$$\Hom_{\mathcal{A}}(F,I)\cong \Hom_{\mathcal{A}^e}(\mathcal{A},\Hom_{\k}(F,I)).$$ Since $\mathcal{A}\in \mathrm{D}^c(\mathcal{A}^e)$, it admits a minimal semi-free resolution $G$, which has a finite semi-basis. We have
\begin{align*}
[\Hom_{\mathcal{A}}(F,I)]^*&\cong[\Hom_{\mathcal{A}^e}(\mathcal{A},\Hom_{\k}(F,I))]^*\\
                           &\simeq  [\Hom_{\mathcal{A}^e}(G,\Hom_{\k}(F,I))]^* \\
                           &\cong  [\Hom_{\mathcal{A}^e}(G,\mathcal{A}^e)\otimes_{\mathcal{A}^e}\Hom_{\k}(F,I)]^*\\
                           &\simeq [\Hom_{\mathcal{A}^e}(G,\mathcal{A}^e)\otimes_{\mathcal{A}^e}(F^*\otimes I)]^* \\
                           &\stackrel{(a)}{\simeq} [F^*\otimes_{\mathcal{A}}\Hom_{\mathcal{A}^e}(G,\mathcal{A}^e)\otimes_{\mathcal{A}}I]^*\\
                           &\cong  \Hom_{\mathcal{A}}(\Hom_{\mathcal{A}^e}(G,\mathcal{A}^e)\otimes_{\mathcal{A}}I, (F^*)^*)\\
                           &\simeq  \Hom_{\mathcal{A}}(\Hom_{\mathcal{A}^e}(G,\mathcal{A}^e)\otimes_{\mathcal{A}}\mathcal{A},(F^*)^*)\\
                           &\cong  \Hom_{\mathcal{A}}(\Hom_{\mathcal{A}^e}(G,\mathcal{A}^e),(F^*)^*)\\
                           &\simeq  \Hom_{\mathcal{A}}(\Sigma^i\mathcal{A}, (F^*)^*)\\
                           &\cong \Sigma^{-i}(F^*)^* \\
                           &\simeq \Sigma^{-i}\k,
\end{align*}
where $(a)$ is obtained by \cite[A1]{Shk}.
So $$H(R\Hom_{\mathcal{A}}(\k,\mathcal{A}))=H(\Hom_{\mathcal{A}}(F,\mathcal{A}))\cong H(\Hom_{\mathcal{A}}(F,I))\cong \Sigma^i\k$$ and $\mathcal{A}$ is Gorenstein.
\end{proof}

It remains to consider the homological properties of non-trivial DG free algebras. We have classified all the isomorphism classes of DG free algebras with two degree one generators. Hence, we only need to check case by case when $n=2$. In \cite{MH}, it is proved that a connected cochain DG algebra $\mathcal{A}$ is a Kozul Calabi-Yau DG algebra if $H(\mathcal{A})$  belongs to either of the following cases:
\begin{align*}
 (a) H(\mathcal{A})\cong \k;  \quad  (b) H(\mathcal{A})= \k[\lceil z\rceil], z\in \mathrm{ker}(\partial_{\mathcal{A}}^1).
\end{align*}
So Proposition \ref{cohomology} indicates that $\mathcal{B}_1,\mathcal{B}_2,\cdots \mathcal{B}_{11}$ and $\mathcal{B}(s,t)$ with $s\in \k^{\times}$ and $st\neq 1$ are all Koszul Calabi-Yau DG algebras except $\mathcal{B}_6$. For $\mathcal{B}_6$, we have the following proposition.
\begin{prop}\label{bsix}
The connected cochain DG algebra $\mathcal{B}_6$ is a Koszul Calabi-Yau DG algebra.
\end{prop}
\begin{proof}
By definition,  $\mathcal{B}^{\#}_6=\k\langle x_1,x_2\rangle$ and the differential $\partial_{\mathcal{B}_6}$ is defined by $$\partial_{\mathcal{B}_6}(x_1)=x_2^2, \partial_{\mathcal{B}_6}(x_2)=0.$$
According to the constructing procedure of the minimal semi-free resolution in \cite[Proposition 2.4]{MW1}, we get a minimal semi-free resolution $f: F\stackrel{\simeq}{\to} {}_{\mathcal{B}_6}\k$, where $F$ is a semi-free DG $\mathcal{B}_6$-module with $$F^{\#}=\mathcal{B}_6^{\#}\oplus \mathcal{B}_6^{\#}\Sigma e_{x_2}\oplus \mathcal{B}_6^{\#}\Sigma e_z, \quad \partial_{F}(\Sigma e_{x_2})=x_2, \quad \partial_{F}(\Sigma e_{z})=x_1+x_2\Sigma e_{x_2};$$ and $f$ is defined by $f|_{\mathcal{B}_6}=\varepsilon$, $f(\Sigma e_{x_2})=0$ and $f(\Sigma e_z)=0$.

We should prove that $f$ is a quasi-isomorphism. It suffices to show $H(F)=\k$.
For any graded cocycle element $a_z\Sigma e_z + a_{x_2}\Sigma e_{x_2}+a\in Z^{2k}(F)$, we have
\begin{align*}
0&=\partial_F(a_z\Sigma e_z + a_{x_2}\Sigma e_{x_2}+a)\\
&=\partial_{\mathcal{B}_6}(a_z)\Sigma e_z +a_z(x_1+x_2\Sigma e_{x_2})+\partial_{\mathcal{B}_6}(a_{x_2})\Sigma e_{x_2}+a_{x_2}x_2+\partial_{\mathcal{B}_6}(a)\\
&=\partial_{\mathcal{B}_6}(a_z)\Sigma e_z +[a_zx_2+\partial_{\mathcal{B}_6}(a_{x_2})]\Sigma e_{x_2}+a_zx_1+a_{x_2}x_2+\partial_{\mathcal{B}_6}(a).
\end{align*}
This implies that
$
\begin{cases}
\partial_{\mathcal{B}_6}(a_z)=0 \\
a_zx_2+\partial_{\mathcal{B}_6}(a_{x_2})=0 \\
a_zx_1+a_{x_2}x_2+\partial_{\mathcal{B}_6}(a)=0.
\end{cases}
$ Since $$Z^{2k}(\mathcal{B}_6)=H^{2k}(\mathcal{B}_6)\oplus B^{2k}(\mathcal{B}_6)$$ and $$H(\mathcal{B}_6)=\k[\lceil x_2\rceil, \lceil x_1x_2+x_2x_1\rceil ]/(\lceil x_2\rceil^2),$$
we have $a_z=\partial_{\mathcal{B}_6}(c)+t(x_1x_2+x_2x_1)^k$ for some $c\in \mathcal{B}_6^{2k-1}, t\in \k$. Then
$$[\partial_{\mathcal{B}_6}(c)+t(x_1x_2+x_2x_1)^k]x_2+\partial_{\mathcal{B}_6}(a_{x_2})=0. $$ Hence $t=0$ and $\partial_{\mathcal{B}_6}(c)x_2+\partial_{\mathcal{B}_6}(a_{x_2})=0$. Then $$a_{x_2}=-cx_2+\partial_{\mathcal{B}_6}(\mu)+s(x_1x_2+x_2x_1)^k$$
for some $\mu\in \mathcal{B}_6^{2k-1}$ and $s\in \k$.
We have
$$\partial_{\mathcal{B}_6}(c)x_1+[-cx_2+\partial_{\mathcal{B}_6}(\mu)+s(x_1x_2+x_2x_1)^k]x_2+\partial_{\mathcal{B}_6}(a)=0,$$
which implies that $\partial_{\mathcal{B}_6}(c)=0, s=0$,
and $\partial_{\mathcal{B}_6}(a)=cx_2^2-\partial_{\mathcal{B}_6}(\mu)x_2$. Then $a_z=0$, $a_{x_2}=-cx_2+\partial_{\mathcal{B}_6}(\mu)$ and $$a=-cx_1-\mu x_2+\partial_{\mathcal{B}_6}(\lambda)+\tau(x_1x_2+x_2x_1)^k$$
 for some $\lambda\in \mathcal{B}_6^{2k-1}, \tau\in \k$. Since $c\in Z^{2k-1}(\mathcal{B}_6)\cong H^{2k-1}(\mathcal{B}_6)\oplus B^{2k-1}(\mathcal{B}_6)$,
 we may let
 $c=\partial_{\mathcal{B}_6}(\chi)+\omega(x_1x_2+x_2x_1)^{k-1}x_2$, for some $\chi\in \mathcal{B}_6^{2k-2}$ and $\omega \in \k$.
Therefore,
\begin{align*}
& a_z\Sigma e_z + a_{x_2}\Sigma e_{x_2}+a\\
=&[-cx_2+\partial_{\mathcal{B}_6}(\mu)]\Sigma e_{x_2}-cx_1-\mu x_2+\partial_{\mathcal{B}_6}(\lambda)+\tau(x_1x_2+x_2x_1)^k\\
=&\{-[\partial_{\mathcal{B}_6}(\chi)+\omega(x_1x_2+x_2x_1)^{k-1}x_2]x_2+\partial_{\mathcal{B}_6}(\mu)\}\Sigma e_{x_2} \\
 & -[\partial_{\mathcal{B}_6}(\chi)+\omega(x_1x_2+x_2x_1)^{k-1}x_2]x_1-\mu x_2+\partial_{\mathcal{B}_6}(\lambda)+\tau(x_1x_2+x_2x_1)^k               \\
 =&\partial_F[(\omega-\tau)(x_1x_2+x_2x_1)^{k-1}x_2\Sigma e_z ]\\
 &+ \partial_F\{[-\chi x_2+\mu-\tau(x_1x_2+x_2x_1)^{k-1}x_1]\Sigma e_{x_2}-\chi x_1\}.
\end{align*}
Hence $H^{2k}(F)=0$, for any $k\in \Bbb{N}$. It remains to show
 $H^{2k-1}(F)=0$, for any $k\in \Bbb{N}$.  Let $a_z\Sigma e_z + a_{x_2}\Sigma e_{x_2}+a\in Z^{2k-1}(F)$, we have
\begin{align*}
0&=\partial_F(a_z\Sigma e_z + a_{x_2}\Sigma e_{x_2}+a)\\
&=\partial_{\mathcal{B}_6}(a_z)\Sigma e_z -a_z(x_1+x_2\Sigma e_{x_2})+\partial_{\mathcal{B}_6}(a_{x_2})\Sigma e_{x_2}-a_{x_2}x_2+\partial_{\mathcal{B}_6}(a)\\
&=\partial_{\mathcal{B}_6}(a_z)\Sigma e_z +[-a_zx_2+\partial_{\mathcal{B}_6}(a_{x_2})]\Sigma e_{x_2}-a_zx_1-a_{x_2}x_2+\partial_{\mathcal{B}_6}(a).
\end{align*}
This implies that
$
\begin{cases}
\partial_{\mathcal{B}_6}(a_z)=0 \\
-a_zx_2+\partial_{\mathcal{B}_6}(a_{x_2})=0 \\
-a_zx_1-a_{x_2}x_2+\partial_{\mathcal{B}_6}(a)=0.
\end{cases}
$ Since $$Z^{2k-1}(\mathcal{B}_6)=H^{2k-1}(\mathcal{B}_6)\oplus B^{2k-1}(\mathcal{B}_6)$$ and $$H(\mathcal{B}_6)=\k[\lceil x_2\rceil, \lceil x_1x_2+x_2x_1\rceil ]/(\lceil x_2\rceil^2),$$
we have $a_z=\partial_{\mathcal{B}_6}(c)+t(x_1x_2+x_2x_1)^{k-1}x_2$ for some $c\in \mathcal{B}_6^{2k-2}, t\in \k$. Then
$$-[\partial_{\mathcal{B}_6}(c)+t(x_1x_2+x_2x_1)^{k-1}x_2]x_2+\partial_{\mathcal{B}_6}(a_{x_2})=0. $$ Then
$$a_{x_2}=cx_2+t(x_1x_2+x_2x_1)^{k-1}x_1+\partial_{\mathcal{B}_6}(\mu)+s(x_1x_2+x_2x_1)^{k-1}x_2$$
for some $\mu\in \mathcal{B}_6^{2k-2}$ and $s\in \k$.
We have
\begin{align*}
0=&-[cx_2+t(x_1x_2+x_2x_1)^{k-1}x_1+\partial_{\mathcal{B}_6}(\mu)+s(x_1x_2+x_2x_1)^{k-1}x_2]x_2\\
&-[\partial_{\mathcal{B}_6}(c)+t(x_1x_2+x_2x_1)^{k-1}x_2]x_1+\partial_{\mathcal{B}_6}(a),
\end{align*}
which implies that $t=0$, $a_z=\partial_{\mathcal{B}_6}(c)$, $a_{x_2}=cx_2+\partial_{\mathcal{B}_6}(\mu)+s(x_1x_2+x_2x_1)^{k-1}x_2$ and $a=cx_1+\mu x_2+s(x_1x_2+x_2x_1)^{k-1}x_1+\partial_{\mathcal{B}_6}(\lambda)+\tau(x_1x_2+x_2x_1)^{k-1}x_2$,
 for some $\lambda\in \mathcal{B}_6^{2k-2}, \tau\in \k$. Hence
\begin{align*}
& a_z\Sigma e_z + a_{x_2}\Sigma e_{x_2}+a\\
=&\partial_{\mathcal{B}_6}(c)\Sigma e_{z} +[cx_2+\partial_{\mathcal{B}_6}(\mu)+s(x_1x_2+x_2x_1)^{k-1}x_2] \Sigma e_{x_2} \\
&+cx_1+\mu x_2+s(x_1x_2+x_2x_1)^{k-1}x_1+\partial_{\mathcal{B}_6}(\lambda)+\tau(x_1x_2+x_2x_1)^{k-1}x_2 \\
 =&\partial_F\{[c+s(x_1x_2+x_2x_1)^{k-1}]\Sigma e_z +[\mu+\tau(x_1x_2+x_2x_1)^{k-1}] \Sigma e_{x_2}+\lambda\}.
\end{align*}
Hence $H^{2k-1}(F)=0$, for any $k\in \Bbb{N}$. Therefore, $f$ is a quasi-isomorphism.

Since $F$ has a semi-basis $\{1, \Sigma e_{x_2},\Sigma e_z\}$ concentrated in degree $0$,  $\mathcal{B}_6$ is a Koszul homologically smooth DG algebra.
By the minimality of $F$, we have $$H(\Hom_{\mathcal{B}_6}(F,\k))=\Hom_{\mathcal{B}_6}(F,\k)= \k\cdot 1^*\oplus \k\cdot(\Sigma e_{x_2})^*\oplus \k \cdot(\Sigma e_z)^*.$$  So the Ext-algebra $E=H(\Hom_{\mathcal{B}_6}(F,F))$  is concentrated in degree $0$.
On the other hand, $$\Hom_{\mathcal{B}_6}(F,F)^{\#}\cong (\k \cdot 1^*\oplus \k \cdot (\Sigma e_{x_2})^*\oplus \k \cdot (\Sigma e_z)^*)\otimes_{\k} F^{\#}$$ is concentrated in degree $\ge 0$. This implies that $E= Z^0(\Hom_{\mathcal{B}_6}(F,F))$.
Since $F^{\#}$ is a free graded $\mathcal{B}_6^{\#}$-module with a basis $\{1,\Sigma e_y,\Sigma e_z\}$ concentrated in degree $0$,
  the elements in  $\Hom_{\mathcal{B}_6}(F,F)^0$ are in one to one correspondence with the matrices in $M_3(\k)$. Indeed, any $f\in \Hom_{\mathcal{B}_6}(F_{\k},F_{\k})^0$ is uniquely determined by
  a matrix $A_f=(a_{ij})_{3\times 3}\in M_3(\k)$ with
$$\left(
                         \begin{array}{c}
                          f(1) \\
                          f(\Sigma e_{x_2})\\
                          f(\Sigma e_z)\\
                         \end{array}
                       \right) =      A_f \cdot \left(
                         \begin{array}{c}
                          1 \\
                          \Sigma e_{x_2}\\
                          \Sigma e_z\\
                         \end{array}
                       \right).  $$
                       We see $f\in  Z^0(\Hom_{\mathcal{B}_6}(F,F)$ if and only if $\partial_{F}\circ f=f\circ \partial_{F}$, if and only if
 $$ A_f\cdot \left(
                         \begin{array}{ccc}
                           0 & 0& 0 \\
                           x_2 & 0 & 0 \\
                           x_1 & x_2 & 0 \\
                         \end{array}
                       \right) =  \left(
                         \begin{array}{cccc}
                           0 & 0& 0 \\
                           x_2 & 0 & 0 \\
                           x_1 & x_2 & 0 \\
                         \end{array}
                       \right) \cdot A_f, $$  which is also equivalent to
                       $$\begin{cases}
                       a_{12}=a_{13}=a_{23}=0\\
                       a_{11}=a_{22}=a_{33}\\
                       a_{21}=a_{32}
                       \end{cases}$$
by a direct computation. Let $$  \mathcal{E}=\{ \left(
                         \begin{array}{ccc}
                           a & 0& 0\\
                           b & a & 0 \\
                           c & b & a \\
                         \end{array}
                       \right)\quad | \quad a,b,c,\in \k \}$$ be the subalgebra of the matrix algebra. Then one sees $E\cong \mathcal{E}$.
                       Set \begin{align*} e_1= \left(
                         \begin{array}{ccc}
                           1 & 0& 0\\
                           0 & 1 & 0 \\
                           0 & 0 & 1 \\
                         \end{array}
                       \right),& e_2= \left(
                         \begin{array}{ccc}
                           0 & 0& 0\\
                           1 & 0 & 0 \\
                           0 & 1 & 0 \\
                         \end{array}
                       \right),
                        e_3= \left(
                         \begin{array}{ccc}
                           0 & 0& 0\\
                           0 & 0 & 0 \\
                           1 & 0 & 0 \\
                         \end{array}
                       \right).
                       \end{align*}
                     Then $\{e_1,e_2,e_3\}$ is a $\k$-linear bases of the $\k$-algebra
                        $\mathcal{E}$. The multiplication on $\mathcal{E}$ is defined by the following relations
                       $$\begin{cases} e_1\cdot e_i=e_i\cdot e_1=e_i, i=1,2,3 \\
                        e_2^2=e_3, e_2\cdot e_3=e_3\cdot e_2=0
                       \end{cases} .$$
                       So $\mathcal{E}$ is a local commutative $\k$-algebra isomorphic to $\k[X]/(X^3)$.  Hence
$E\cong \k[X]/(X^3)$ is a symmetric Frobenius algebra concentrated
in degree $0$. This implies that
$\mathrm{Tor}_{\mathcal{B}_6}^0(\k_{\mathcal{B}_6},{}_{\mathcal{B}_6}\k)\cong
E^*$ is a symmetric coalgebra. By \cite[Theorem 4.2]{HM1},
$\mathcal{B}_6$ is a Koszul Calabi-Yau DG algebra.

\end{proof}

By Proposition \ref{nonsym}, Proposition \ref{symone}, Proposition \ref{symtwo}, Proposition \ref{stone}, Proposition \ref{cohomology} and Proposition \ref{bsix}, we reach the following conclusion.
\begin{thm}\label{finresult}
Let $\mathcal{A}$ be a DG free algebra with $2$ degree one generators. Then $\mathcal{A}$ is a Koszul Calabi-Yau DG algebra if and only if $\partial_{\mathcal{A}}\neq 0$.
\end{thm}

\section*{Acknowledgments}
 The first author is
supported by NSFC  (Grant No.11001056),
the
China Postdoctoral
Science Foundation  (Grant Nos.
20090450066 and 201003244),  the
Key Disciplines of Shanghai Municipality (Grant No.S30104) and the Innovation Program of
Shanghai Municipal Education Commission (Grant No.12YZ031).

\def\refname{References}

\end{document}